\RequirePackage{fix-cm}
\documentclass{svjour3} 
\smartqed 
\usepackage{tabularx,booktabs}
\usepackage{color}
\usepackage{lscape}
\usepackage{latexsym}
\usepackage{rotating}
\usepackage{amsfonts,amssymb,enumerate}
\usepackage{threeparttable}
\usepackage{graphicx}
\usepackage{multirow}
\usepackage[ruled, vlined, algonl]{algorithm2e}
\usepackage{bm}
\usepackage{galois}
\usepackage{amsmath}

\renewcommand{\thetable}{\Roman{table}}
\counterwithout{equation}{section}
\newenvironment{proof}{\noindent{\em Proof:}}{\quad \hfill$\Box$\vspace{2ex}}

\DeclareMathOperator*{\argmin}{arg\, min}

\DeclareMathOperator*{\sgn}{sgn}
\DeclareMathOperator*{\dom}{dom}

\newcommand{\mtrx}[1]{\mathsf{#1}}

\newcommand{\mA}{\mtrx{A}}
\newcommand{\mB}{\mtrx{B}}
\newcommand{\mD}{\mtrx{D}}
\newcommand{\mE}{\mtrx{E}}

\newcommand{\mH}{\mtrx{H}}
\newcommand{\mI}{\mtrx{\mathrm{I}}}

\newcommand{\mP}{\mtrx{P}}

\newcommand{\vect}[1]{\bm{#1}}

\newcommand{\va}{\vect{a}}
\newcommand{\vb}{\vect{b}}
\newcommand{\vc}{\vect{c}}
\newcommand{\vd}{\vect{d}}
\newcommand{\ve}{\vect{e}}

\newcommand{\vp}{\vect{p}}
\newcommand{\vq}{\vect{q}}
\newcommand{\vs}{\vect{s}}

\newcommand{\vu}{\vect{u}}
\newcommand{\vv}{\vect{v}}
\newcommand{\vw}{\vect{w}}
\newcommand{\vx}{\vect{x}}
\newcommand{\vy}{\vect{y}}

\newcommand{\vz}{\vect{z}}

\newcommand{\vzero}{\vect{0}}

\usepackage{enumitem, xcolor, arydshln}
\usepackage{caption}
\usepackage{subcaption}

\newcommand{\R}{\mathbb{R}}

\newenvironment{bluetext}{\color{blue}}{\ignorespacesafterend}

\begin{document}

\title{Sparse Recovery: The Square of $\ell_1/\ell_2$ Norms 
}
\author{Jianqing Jia         \and
       Ashley Prater-Bennette \and
       Lixin Shen           \and
       Erin E. Tripp
}

\institute{Jianqing Jia \at Department of Mathematics, Syracuse University, Syracuse, NY 13244. 
              \email{jjia10@syr.edu}           
           \and
           Ashley Prater-Bennette \at Air Force Research Laboratory, Rome, NY 13441. 
           \email{ashley.prater-bennette@us.af.mil}           
           \and
           Lixin Shen \at   Department of Mathematics, Syracuse University, Syracuse, NY 13244. 
              \email{lshen03@syr.edu}           
           \and
           Erin E.Tripp \at Hamilton College, Clinton, NY 13323. 
           \email{etripp@hamilton.edu}     
}

\date{Received: date / Accepted: date}

\maketitle

\begin{abstract}
This paper introduces a nonconvex approach for sparse signal recovery, proposing a novel model termed the $\tau_2$-model, which utilizes the squared $\ell_1/\ell_2$ norms for this purpose. Our model offers an advancement over the $\ell_0$ norm, which is often computationally intractable and less effective in practical scenarios. Grounded in the concept of effective sparsity, our approach robustly measures the number of significant coordinates in a signal, making it a powerful alternative for sparse signal estimation. The $\tau_2$-model is particularly advantageous due to its computational efficiency and practical applicability.

We detail two accompanying algorithms based on Dinkelbach’s procedure and a difference of convex functions strategy. The first algorithm views the model as a linear-constrained quadratic programming problem in noiseless scenarios and as a quadratic-constrained quadratic programming problem in noisy scenarios. The second algorithm, capable of handling both noiseless and noisy cases, is based on the alternating direction linearized proximal method of multipliers. We also explore the model's properties, including the existence of solutions under certain conditions, and discuss the convergence properties of the algorithms. Numerical experiments with various sensing matrices validate the effectiveness of our proposed model.
\end{abstract}

\section{Introduction}
The problem of compressive sensing is to estimate an unknown sparse signal $\vx \in \mathbb{R}^n$ from $m$ linear measurements $\vb\in \mathbb{R}^m$ given by
$$
\vb=\mA \vx +\vz,
$$
where $\mA \in \mathbb{R}^{m \times n}$ is a sensing matrix, $m$ is much smaller than the signal dimension $n$, and $\vz$ is a stochastic or deterministic unknown error term. Mathematically, this problem can be formulated as
\begin{equation}\label{model:ell0}
\min \{\|\vx\|_0: \|\mA\vx-\vb\|_2 \le \epsilon, \; \vx \in \mathbb{R}^n\},
\end{equation}
where the quantity $\|\vx\|_0$, i.e.,  the $\ell_0$ norm of $\vx$, denotes the number of nonzeros in $\vx$, and $\epsilon$ bounds the amount of noise in the data $\vb$. The $\ell_0$ norm of $\vx$ causes model~\eqref{model:ell0} to be an NP-hard problem, and it is also not a useful measure of the significant number of entries in $\vx$.  Therefore, many alternative models have been proposed to replace the $\ell_0$ norm by sparsity promoting functions and solve the resulting models in computationally tractable algorithms. Examples of the sparsity promoting functions used include the $\ell_1$ norm \cite{Candes-Romberg-Tao:IEEE-TIT:06}, the $\ell_p$ quasi norm ($0<p< 1$) \cite{Chartrand:IEEE-Letter:07,Chen-Shen-Suter:IET:16,Donoho:IEEEIT:06,Prater-Shen-Suter:CSDA:15}, the minimax concave penalty function \cite{Zhang:AS:2010}, the log-sum penalty function \cite{Prater-Shen-Tripp:JSC:2022}, and its generalized form as discussed in \cite{Shen-Suter-Tripp:JOTA:2019}. 

A practical drawback of the $\|\vx\|_0$ norm was highlighted in \cite{Lopes:IEEEIT:2016}, as it is highly sensitive to small entries of $\vx$. In light of this observation, a notion of effective sparsity was introduced to address this issue. Effective sparsity aims to quantify the ``effective number of coordinates of $\vx$'' while remaining robust against small perturbations. For any nonzero vector $\vx\in \R^p$, an induced distribution $\pi(\vx)\in \R^p$ is defined over the index set $\{1, \ldots, p\}$, assigning mass $\pi_j(\vx):=|x_j|/\|\vx\|_1$ to index $j$. It is worth noting that if $\vx$ is sparse, the resulting distribution $\pi(\vx)$ exhibits low entropy. The effective sparsity measure $\tau_q(\vx)$ is then defined as
$$
\tau_q(\vx):=\mathrm{exp}(H_q(\pi(\vx))),
$$
where $H_q(\vx)=\frac{1}{1-q}\mathrm{log}(\sum_{i=1}^p \pi_i(\vx)^q)$ is the R\'enyi entropy of order $q\in (0, \infty) \setminus \{1\}$. For $\vx\ne \mathbf{0}$ and $q\notin \{0, 1, \infty\}$, the effective sparsity can be conveniently expressed as
$$
\tau_q(\vx)=\left(\frac{\|\vx\|_q}{\|\vx\|_1}\right)^{\frac{q}{1-q}}.
$$
As with $H_q$, the cases of $q \in \{0, 1, \infty\}$ are evaluated as limits, which is feasible and informative from the viewpoint of information theory.

This family of entropy-based sparsity measures in terms of R\'enyi entropy gives a conceptual foundation for several norm ratios that have appeared elsewhere in the sparsity literature. The choice of $q$ is relevant to many considerations. Among all choices, the case of $q=2$ 
turns out to be attractive, for example, \cite{lopes2013estimating,Lopes:IEEEIT:2016}
showed that $\tau_2$, referred to as the square of $\ell_1/\ell_2$ norms,  plays an intuitive role in the performance of the Basis Pursuit Denoising algorithm. This sparsity measure was employed to relax the necessary and sufficient conditions for exact $\ell_1$ recovery \cite{Tang-Nehorai:IEEESP:2011}. 

Among various nonconvex regularizers, the ratio of the $\ell_1$ and $\ell_2$ norms, denoted by $\ell_1/\ell_2$, as a sparsity measure was introduced in \cite{Hoyer:Proc-IEEENNSP:2002} and was further investigated in \cite{Hurley-Rickard:IEEEIT:2009}. A model that uses $\ell_1/\ell_2$ to replace the $\ell_0$ norm in model~\eqref{model:ell0} was recently studied in \cite{Rahimi-Wang-Dong-Lou:SIAMSC:2019,Wang-Yan-Rahimi-Lou:IEEESP:2020,Yin2-Esser-Xin:CIS2014}. This unconstrained model, referred to as the \textbf{$\mathbf{\ell_1/\ell_2}$-model}, is written as:
\begin{equation}\label{model:square-root}
\arg\inf\left\{\frac{\|\vx\|_1}{\|\vx\|_2}: \|\mA \vx-\vb\|_2 \le \epsilon, \; \vx\in \mathbb{R}^n\right\}.
\end{equation}
Some theoretical results about the existence of the solutions to model~\eqref{model:square-root} and the Kurdyka-Lojasiewicz property of the cost function of the model are investigated in \cite{Zeng-Yi-Pong:SIAMOP:2021}. 

In our approach, we consider a model that replaces the $\ell_0$ norm in model~\eqref{model:ell0} with $\tau_2$, resulting in the following formulation:
\begin{equation}\label{model:square}
\arg\inf\left\{\tau_2(\vx)=\frac{\|\vx\|_1^2}{\|\vx\|_2^2}:\|\mA \vx-\vb\|_2 \le \epsilon, \; \vx\in \mathbb{R}^n\right\},
\end{equation}
referred to as the \textbf{$\mathbf{\tau_2}$-model}.  The solution sets of the $\tau_2$-model and the $\ell_1/\ell_2$-model are clearly identical in terms of their solutions. Due to the nonconvexity of $\ell_1/\ell_2$ and $\tau_2$, the numerical solutions to both models could be different. Many algorithms have been developed for the $\ell_1/\ell_2$-model \cite{Rahimi-Wang-Dong-Lou:SIAMSC:2019,Wang-Yan-Rahimi-Lou:IEEESP:2020,Zeng-Yi-Pong:SIAMOP:2021}. However, we are not aware of the research work on directly solving the $\tau_2$-model. In the paper, we will focus on developing algorithms for solving the $\tau_2$-model. 

Our approach for the $\tau_2$-model is grounded in Dinkelbach's procedure, an iterative approach used to solve fractional programming problems \cite{Crouzeix-Ferland-Schaible:JOTA-1985}.  Dinkelbach's procedure systematically converts a fractional programming problem into an equivalent non-fractional form. This conversion involves introducing a new variable and reformulating the problem as a minimization problem,  where the objective function is the difference between the numerator and denominator of the fractional programming objective function, scaled by the introduced parameter. The optimal value, expressed as a function of the parameter, defines the Dinkelbach-induced function. The iterative steps of Dinkelbach's method entail solving this transformed problem, adjusting the parameter value in each iteration until convergence. The procedure concludes with the determination of the optimal parameter value corresponding to the solution of the original fractional programming problem. For the $\tau_2$-model, we carefully study the properties of its Dinkelbach-induced function and highlight that the minimization problem associated with the Dinkelbach-induced function is a quadratic programming task. Notably, we believe to be the first to develop  numerical methods specifically tailored for the $\tau_2$-model.

This paper is organized as follows. In Section~\ref{sec:prel}, we review the Dinkelbach procedure for fractional programming. Section~\ref{sec:ourapproach} is dedicated to the theoretical analysis of the $\tau_2$-model with $\epsilon=0$, including the relationship between the $\tau_2$-model and the root of its Dinkelbach-induced function, as well as the solutions existence of the $\tau_2$-model. In Section~\ref{sec:QP}, we highlight that the value of Dinkelbach-induced function is the optimal value of a linear-constrained quadratic programming (LCQP) problem when $\epsilon=0$ and a  quadratic-constrained quadratic programming
(QCQP) problem when $\epsilon>0$.  In Section~\ref{sec:convergence}, we develop an algorithm based on the alternating direction linearized proximal method of multipliers,  capable of handling both noiseless and noisy cases. We also present the convergence analysis of the algorithm. Section~\ref{sec:numerical} focuses on presenting the results of numerical experiments designed to showcase the efficacy of the proposed algorithm. Finally, we summarize our findings and draw conclusions in Section~\ref{sec:conclusion}.

\section{Dinkelbach's Procedure}\label{sec:prel}
In this study, we designate $\mathbb{R}^n$ to denote the Euclidean space of dimension $n$. The bold lowercase letters, such as $\vx$, signify vectors, with the $j$th component represented by the corresponding lowercase letter $x_j$. The notation $\mathrm{supp}(\vx)$ denotes the support of the vector $\vx$, defined as $\mathrm{supp}(\vx)=\{k: x_k \neq 0\}$. The matrices are indicated by bold uppercase letters such as $\mA$ and $\mB$. For a set $S \subset \mathbb{R}^n$, $\mathrm{cl}(S)$ denotes the closure of $S$.

The $\ell_p$ norm of $\vx=[x_1,\ldots, x_n]^\top \in \mathbb{R}^n$ is defined as $\|\vx\|_p=(\sum_{k=1}^n |x_k|^p)^{1/p}$ for $1\le p <\infty$, $\|\vx\|_\infty=\max_{1\le k \le n} |x_k|$, and $\|\vx\|_0$ being the number of nonzero components in $\vx$.

Both the $\tau_2$-model and the $\ell_1/\ell_2$-model are two special examples of traditional fractional programming. A fractional programming problem is defined by
\begin{equation}\label{model-Trad-FP}
\bar{\alpha} = \inf\left\{\frac{f(\vx)}{g(\vx)}: \vx\in S \right\},
\end{equation}
where $S$ is a nonempty subset of $\mathbb{R}^n$, $f$ and $g$ are continuous on an open set $\tilde{S}$ in $\mathbb{R}^n$ including $\mathrm{cl}(S)$, and $g(x)>0$ for all $x \in \tilde{S}$. When we identify $S=\{\vx: \|\mA \vx - \vb\|_2 \le \epsilon\}$, $\tilde{S}=\mathbb{R}^n \setminus \{\mathbf{0}\}$, a fractional programming \eqref{model-Trad-FP} becomes the $\tau_2$-model if $f(\vx)=\|\vx\|_1^2$ and $g(\vx)=\|\vx\|_2^2$, while it becomes the $\ell_1/\ell_2$-model if $f(\vx)=\|\vx\|_1$ and $g(\vx)=\|\vx\|_2$ .

One of the classical methods to handle model~\eqref{model-Trad-FP} is Dinkelbach's procedure which is related to the following auxiliary problem with a parameter $\alpha$
\begin{equation}\label{model-Trad-FP-alpha}
F(\alpha) = \inf\left\{f(\vx)-\alpha g(\vx): \vx\in S \right\}.
\end{equation}
This function $F$ is known as the Dinkelbach-induced function of model \eqref{model-Trad-FP}, or simply the Dinkelbach-induced function when the model is clear from the context.  In \cite{Crouzeix-Ferland-Schaible:JOTA-1985} it has been shown that if model \eqref{model-Trad-FP} has an optimal solution at $\bar{\vx} \in S$, then this solution is also optimal for \eqref{model-Trad-FP-alpha}, and the optimal objective value of the latter is zero. Conversely, if \eqref{model-Trad-FP-alpha} has $\bar{\vx} \in S$ as an optimal solution and its optimal objective value is zero, then $\bar{\vx}$ is also an optimal solution for \eqref{model-Trad-FP}. The algorithm arising from Dinkelbach's procedure is described as follows:
\begin{enumerate}
    \item Start with some $\vx^{(0)} \in S$.
    \item Set $\alpha^{(1)} = f(\vx^{(0)})/g(\vx^{(0)})$.
    \item Update $\alpha^{(k+1)} = f(\vx^{(k)})/g(\vx^{(k)})$, where  $\vx^{(k)}$ is assumed to be an optimal solution of the corresponding problem. 
\end{enumerate}
This procedure can be viewed as the Newton method for finding a root of the equation 
$F(\alpha)=0$, as discussed in \cite{Crouzeix-Ferland:MP-1991,Ibaraki:MP-1983}.

The work in \cite{Wang-Yan-Rahimi-Lou:IEEESP:2020} for the $\ell_1/\ell_2$-model with $\epsilon=0$ essentially follows Dinkelbach's procedure. In this situation, the objective function in the $k$th iteration is $\|\vx\|_1-\alpha^{(k)} \|\vx\|_2$ which is the difference of two convex functions and is provably unbounded from below. To overcome this difficulty and accelerate the optimization problem at the $k$th iteration, the term $\alpha^{(k)} \|x\|_2$ is replaced by its linearization at the previous iterate $\vx^{(k-1)}$ with an additional regularization term in \cite{Wang-Yan-Rahimi-Lou:IEEESP:2020}.

The objective of this paper is to scrutinize the properties and devise algorithms for the $\tau_2$-model within the framework of Dinkelbach's procedure.

\section{Theoretical Analysis for the $\tau_2$-Model with $\epsilon=0$}\label{sec:ourapproach}
In this section, we will present the special properties of the $\tau_2$-model with $\epsilon=0$ and the associated optimization problem arising from Dinkelbach's procedure. In this situation,  we rewrite the $\tau_2$-model~\eqref{model:square} as
\begin{equation*}\label{model:square-new}
\bar{\alpha}=\inf\left\{\frac{\|\vx\|_1^2}{\|\vx\|_2^2}:\mA \vx= \vb, \; \vx\in \mathbb{R}^n\right\}. \tag{$P$}
\end{equation*}
We then define the Dinkelbach-induced function of \eqref{model:square-new} as:
\begin{equation}\label{model-Trad-FP-square}
F(\alpha) = \inf\left\{\|\vx\|_1^2-\alpha \|\vx\|_2^2: \mA \vx= \vb, \; \vx\in \mathbb{R}^n \right\}. \tag{$Q$}
\end{equation}
In the following discussion, we always assume that $\mA$ is the full row rank, that the vector $\vb$ is nonzero and is in the range space of $\mA$.

From \cite[Proposition 2.1]{Crouzeix-Ferland-Schaible:JOTA-1985}, we have the following results regarding $\bar{\alpha}$ and $F(\alpha)$ in \eqref{model:square-new} and \eqref{model-Trad-FP-square}:
\begin{itemize}
\item[(i)] $F(\alpha)<+\infty$; $F$ is nonincreasing and upper semicontinuous;
\item[(ii)] $F(\alpha)<0$ if and only if $\alpha>\bar{\alpha}$; hence $F(\bar{\alpha}) \ge 0$;
\item[(iii)] If \eqref{model:square-new} has an optimal solution, then $F(\bar{\alpha}) = 0$;
\item[(iv)] If $F(\bar{\alpha}) = 0$, then  \eqref{model:square-new} and  \eqref{model-Trad-FP-square} have the same set of optimal solutions (which may be empty).
\end{itemize}

Some specific properties of Problems~\eqref{model:square-new} and \eqref{model-Trad-FP-square} will be elucidated in the following subsections. 

 \subsection{Properties for Problems~\eqref{model:square-new} and \eqref{model-Trad-FP-square}}

We begin with defining a parameter
\begin{equation}\label{def:alpha-star}
\alpha^* = \min_{\|\vv\|_2=1, \; \mA\vv=\bm{0}}\|\vv\|_1^2.
\end{equation}
which plays an important role in the analysis of the behavior of $F(\alpha)$ in \eqref{model-Trad-FP-square} and $\bar{\alpha}$ in \eqref{model:square-new}. Since the set $\{\vv: \|\vv\|_2=1, \; \mA\vv=\bm{0}\}$ is compact and the $\ell_1$ norm is continuous, the optimal value $\alpha^*$ is achievable at some unit vector in $\mathrm{Ker} (\mA)$.

For any $\vx\in \R^n$, the inequality $\|\vx\|_2^2 \le \|\vx\|_1^2 \le n \|\vx\|_2^2$ holds. Consequently, $F(\alpha) \ge 0$ for all $\alpha \le 1$ and $F(\alpha) \le 0$ for all $\alpha \ge n$. As discussed above, identifying a root of the equation $F(\alpha)=0$ is crucial in Dinkelbach’s procedure. Therefore, we focus our analysis on the behavior of $F$ within the interval $[1, n]$ and present the following result.

\begin{proposition}\label{prop:F(alpha)}
Let $\alpha^*$ be a number given in \eqref{def:alpha-star} and let $F$ be the Dinkelbach-induced function of \eqref{model:square-new} defined in \eqref{model-Trad-FP-square}. The following statements hold.
\begin{itemize}
\item[(i)] If $\alpha^*=1$ or $\alpha^*=n$, then $F(\alpha^*)$ is a real number.
\item[(ii)] If $1<\alpha^*<n$,  then the function $F$ is finite and strictly decreasing on $[1, \alpha^*)$, and takes the value of $-\infty$ on $(\alpha^*, n]$.
\end{itemize}

\end{proposition}
\begin{proof}\ \
Let $\vx_0$ be a solution of the system $\mA \vx=\vb$ such that $\langle \vx_0, \vv \rangle =0$ for all $\vv \in \mathrm{Ker}(\mA)$. Then, the solution set of this system is
$$
\{\vx: \mA \vx=\vb\}=\vx_0+\mathrm{Ker}(\mA)=\{\vx_0+t \vv: \|\vv\|_2=1, \mA \vv =\bm{0}, t \in \mathbb{R} \}
$$
with $\vx_0 \perp \mathrm{Ker}(\mA)$. With these preparations, the function $F$ in \eqref{model-Trad-FP-square} can be written as
$$
F(\alpha)=\inf_{\vv\in \mathbb{R}^n, t \in \mathbb{R}}\{\|\vx_0+t\vv\|_1^2-\alpha\|\vx_0+t \vv\|_2^2:\;\; \mA \vv= \bm{0}, \|\vv\|_2=1\}.
$$
Hence, $F(\alpha)$ being finite or negative infinity depends on the behavior of the objective function of the above optimization problem for large value of $t$. To this end, for a given unit vector $\vv\in \mathrm{Ker}(\mA)$, $\vx_0$ with $\mA \vx_0=\vb$, and $\alpha \in [1,n]$, define $K: \mathbb{R} \rightarrow \mathbb{R}$ as
$$
K(t)=\|\vx_0+t\vv\|_1^2-\alpha\|\vx_0+t \vv\|_2^2.
$$
Further, for the given $\vx_0$  and  the unit vector $\vv\in \mathrm{Ker}(\mA)$, define a subset of $\mathbb{R}$ as
$$
\mathcal{S}:=(-\infty, -\sigma_{\vx_0,\vv}) \cup (\sigma_{\vx_0,\vv},\infty)
$$
with $\sigma_{\vx_0,\vv}:=\min\{|s|: |s v_i| \ge |x_{0,i}|, i \in \mathrm{supp}(\vv)\}$. 

By the continuity of $K$ on the closed interval $\mathcal{S}^c =[-\sigma_{\vx_0,\vv},\sigma_{\vx_0,\vv}]$, the function $K$ can achieve its global minimum on this interval. Away from the interval, i.e., for $t \in \mathcal{S}$, we have
\begin{eqnarray*}
    \|\vx_0+t\vv\|_1&=& \underbrace{\sum_{i \in \mathrm{supp}(\vv)}tv_i \sgn(tv_i)}_{=|t|\|\vv\|_1} + \underbrace{\sum_{i \in \mathrm{supp}(\vv)}x_{0,i}\sgn(tv_i) + \sum_{i \notin \mathrm{supp}(\vv)} |x_{0,i}|}_{q(t):=}\\
    &=&|t|\|\vv\|_1 + q(t)
\end{eqnarray*}
and
$$
\|\vx_0+t\vv\|_2^2 = \|\vx_0\|_2^2 + t^2,
$$
which leads to
$$
K(t)=(\|\vv\|_1^2-\alpha)t^2 + 2 q(t)\|\vv\|_1 |t| +(q(t))^2-\alpha\|\vx_0\|_2^2
$$
for $t \in \mathcal{S}$. Noting that $|q(t)| \le \|\vx_0\|_1$ for all $t$,  we conclude from the above discussions that
$$
\inf_{t \in \mathbb{R}} K(t)=
\left\{
  \begin{array}{ll}
    \mbox{a real number}, & \hbox{if $\|\vv\|_1^2 >\alpha$;} \\
    \mbox{a real number}, & \hbox{if $\|\vv\|_1^2 =\alpha$ and $q(t) \ge 0$;} \\    
    -\infty, & \hbox{otherwise.}
  \end{array}
\right.
$$
Thus,
$$
F(\alpha)=
\left\{
  \begin{array}{ll}
    \mbox{a real number}, & \hbox{if $\alpha^* >\alpha$;} \\
    \mbox{indefinite}, & \hbox{if $\alpha^* =\alpha$;}\\
    -\infty, & \hbox{if $\alpha^* <\alpha$.}
    
  \end{array}
\right.
$$

Item (i): We notice that $1 \le \alpha^* \le n$. Let $\vv \in \mathrm{Ker}(\mA)$ with $\|\vv\|_2=1$ such that $\alpha^*=\|\vv\|_1^2$. If $\alpha^*=1$, then the vector $\vv$ must have only nonzero element with the value of $\pm 1$. Without loss of generality, we assume that $v_1=\pm 1$. It leads to $x_{0,1}=0$ because of $\langle \vv, \vx_0 \rangle=v_1x_{0,1}=0$. As a result, $q(t)=\sum_{i=2}^n |x_{0,i}| \ge 0$, hence, $F(1)$ is a real number.

On the other hand, if $\alpha^*=n$, then every component of the vector $\vv$ must be $\pm \frac{1}{\sqrt{n}}$. Since $\sgn(v_i)=\sqrt{n} v_i$, we have
$$
q(t)=\sum_{i=1}^n x_{0,i}\sgn(tv_i)=\sgn(t)\sum_{i=1}^n x_{0,i}\sgn(v_i)=\sgn(t)\sqrt{n}\sum_{i=1}^n x_{0,i}v_i=0.
$$
Then, $F(n)$ is a real number if $\alpha^*=n$.

Item (ii): If $1< \alpha^* <n$, we have that $F([1, \alpha^*)) \subset \mathbb{R}$ and $F((\alpha^*, n]) = \{-\infty\}$. Next, we show the strictly decreasing on $[1, \alpha^*)$. Assume that both $\alpha_1$ and $\alpha_2$ are in $[1, \alpha^*)$ with $\alpha_1 < \alpha_2$. Then, there exist $\vp$ and $\vq$ with $\mA \vp =\mA \vq=\vb$ such that $F(\alpha_1)=\|\vp\|_1^2-\alpha_1\|\vp\|_2^2$ and $F(\alpha_2)=\|\vq\|_1^2-\alpha_2\|\vq\|_2^2$. Hence
\begin{eqnarray*}
F(\alpha_1)&=&\|\vp\|_1^2-\alpha_1\|\vp\|_2^2 \\
&=& \|\vp\|_1^2-\alpha_2\|\vp\|_2^2 +(\alpha_2-\alpha_1)\|\vp\|_2^2 \\
&\ge&\|\vq\|_1^2-\alpha_2\|\vq\|_2^2 +(\alpha_2-\alpha_1)\|\vp\|_2^2 \\
&>&F(\alpha_2).
\end{eqnarray*}
That is, $F$ is strictly decreasing on $[1, \alpha^*)$.
\end{proof}

Two comments regarding Proposition~\ref{prop:F(alpha)} are warranted. The initial observation notes that the finiteness of $F(1)$ when $\alpha^*=1$ can be inferred from the fact that $\|\vx\|_1^2-\|\vx\|_2^2 \ge 0$ holds for all $\mA \vx = \vb$. The second comment pertains to the indefiniteness of $F(\alpha^*)$, specifically, whether $F(\alpha^*)$ belongs to the set of real numbers $\mathbb{R}$ or if $F(\alpha^*)$ equals negative infinity. To delve into this matter, two examples will be presented.

\begin{example}
The first example is to show the case $F(\alpha^*)=-\infty$. Define
\begin{equation*}
    \mA:=\begin{bmatrix}
1 & -1 & 0 & 0 & 0 & 0 \\
1 & 0 & -1 & 0 & 0 & 0 \\
0 & 1 & 1 & 1 & 0 & 0 \\
2 & 2 & 0 & 0 & 1 & 0 \\
1 & 1 & 0 & 0 & 0 & -1
\end{bmatrix} \quad \mbox{and} \quad
\vb=\begin{bmatrix}0 \\ 0 \\ 20 \\ 40 \\ 18\end{bmatrix}.
\end{equation*}
This matrix $\mA$ and vector $\vb$ are borrowed from \cite{Rahimi-Wang-Dong-Lou:SIAMSC:2019}. Then,
$$
\{\vx: \mA \vx=\vb\} = \vx_0 + \mathrm{Ker}(\mA) = \vx_0 + \mathrm{Span}(\vv),
$$
where
$$
\vx_0=\begin{bmatrix}0\\ 0 \\ 0 \\ 20 \\ 40 \\ -18\end{bmatrix}
+\frac{236}{27}\begin{bmatrix}1\\ 1 \\ 1 \\ -2 \\ -4\\ 2\end{bmatrix}, \quad
\vv=\frac{\sqrt{3}}{9}\begin{bmatrix}1\\ 1 \\ 1 \\ -2 \\ -4\\ 2\end{bmatrix}.
$$
We can verify that $\mA \vx_0=\vb$ and $\langle \vx_0, \vv\rangle=0$.
For this example, we have $\alpha^* = \|\vv\|_1^2 = \frac{121}{27}$.
Since $\mathrm{supp}(\vv)=\{1,2,3,4,5,6\}$, we have $q(t) = \frac{490}{27}\sgn(t)$. Here, the function $q(t)$ is defined in the proof of Proposition~\ref{prop:F(alpha)}. This implies that $F(\alpha^*)=-\infty$. 

For case $F(\alpha^*)=-\infty$, when \eqref{model:square-new} has an optimal solution $\bar{\alpha}$, we should have $\bar{\alpha}< \alpha^*$. Here we can further compute the value of $\bar{\alpha}$. Indeed, we have
$$
\bar{\alpha} = \inf_{t \in \mathbb{R}}\frac{\|\vx_0+t\vv\|_1^2}{\|\vx_0+t\vv\|_2^2} = \frac{\|\vx_0-\frac{236}{3\sqrt{3}}\vv\|_1^2}{\|\vx_0-\frac{236}{3\sqrt{3}}\vv\|_2^2}=\frac{1521}{581},
$$
which is strictly smaller than $\alpha^*$.
\end{example}

The matrix $\mA$ and the vector $\vb$ in the subsequent example are derived from those in the preceding example by omitting their second rows.

\begin{example}
The second example we consider here has
\begin{equation*}
    \mA:=\begin{bmatrix}
1 & -1 & 0 & 0 & 0 & 0 \\
0 & 1 & 1 & 1 & 0 & 0 \\
2 & 2 & 0 & 0 & 1 & 0 \\
1 & 1 & 0 & 0 & 0 & -1
\end{bmatrix} \quad \mbox{and} \quad
\vb=\begin{bmatrix}0  \\ 20 \\ 40 \\ 18\end{bmatrix}.
\end{equation*}
Then,
$$
\{\vx: \mA \vx=\vb\} = \vx_0 + \mathrm{Ker}(\mA) = \vx_0 + \mathrm{Span}(\vv_1,\vv_2),
$$
where
$$
\vx_0=\begin{bmatrix}9\\ 9 \\ 11 \\ 0 \\ 4 \\ 0\end{bmatrix}-
\frac{1}{45}\begin{bmatrix}-7 \\ -7 \\ 251 \\ -244 \\ 28 \\ -14\end{bmatrix}, \quad
\vv_1=\frac{1}{\sqrt{2}}\begin{bmatrix}0\\ 0 \\ -1 \\ 1 \\ 0 \\ 0\end{bmatrix}, \quad
\vv_2=\frac{1}{\sqrt{90}}\begin{bmatrix}2\\ 2 \\ -1 \\ -1 \\ -8 \\ 4\end{bmatrix}.
$$
We can check that $\{\vv_1,\vv_2\}$ is the orthonormal basis of $\mathrm{Ker}(\mA)$ and $\vx_0$ is orthogonal to $\mathrm{Ker}(\mA)$.

Every unit vector in $\mathrm{Ker}(\mA)$ has a form of $(\cos \theta) \vv_1+(\sin \theta) \vv_2$ with some $\theta\in [0, 2\pi]$. Hence
\begin{eqnarray*}
\alpha^*&=&\min_{\|\vv\|_2=1, \; \mA\vv=\bm{0}}\|\vv\|_1^2 \\
&=&\min_{0 \le \theta \le 2\pi} \|(\cos \theta) \vv_1+(\sin \theta) \vv_2\|_1^2 \\
&=&\min_{0 \le \theta \le 2\pi} \left(\frac{16}{\sqrt{90}}|\sin \theta|+\left|\frac{\cos \theta}{\sqrt{2}}+\frac{\sin\theta}{\sqrt{90}}\right|+\left|\frac{\cos \theta}{\sqrt{2}}-\frac{\sin\theta}{\sqrt{90}}\right|\right)^2 \\
&=&2.
\end{eqnarray*}
In this situation, we have $\mathrm{supp}(\vv)=\{3,4\}$ and
$$
q(t)= \left(11-\frac{251}{45}\right)\sgn(-t)+\frac{244}{45}\sgn(t)+2\left(9+\frac{7}{45}\right)+\left(4-\frac{28}{45}\right)+14 =35\frac{31}{45}>0.
$$
Once again, $q(t)$ is defined within the context of proof of Proposition~\ref{prop:F(alpha)}. Hence, $F(\alpha^*)$ is a real number.  

Now, for case $F(\alpha^*)\in \R$, when \eqref{model:square-new} has an optimal solution $\bar{\alpha}$, we should have $\bar{\alpha}\le \alpha^*$. Next, we can compute the value of $\bar{\alpha}$.
$$
\bar{\alpha} = \inf_{(s,t) \in \mathbb{R}^2}\frac{\|\vx_0+s\vv_1+t\vv_2\|_1^2}{\|\vx_0+s\vv_1+t\vv_2\|_2^2} = 2.
$$
The value of $\bar{\alpha}$ is obtained by the unbounded minimizing sequence $\{\vx_0+k\vv_1\}_{k=1}^\infty$. In this case, $\bar{\alpha}=\alpha^*$.
\end{example}

With the help of the two detailed examples, we have addressed the issue of the indefiniteness of $F(\alpha)$ for $\alpha\ne 1$ and $n$. This has shed light on whether $F(\alpha^*)$ is a real number or $-\infty$ for different scenarios. Furthermore, the examples provide a practical illustration of the following proposition, which establishes a connection between the optimal value of $\bar{\alpha}$ in model~\eqref{model:square-new} and the parameter $\alpha^*$ in \eqref{def:alpha-star}.

\begin{proposition}\label{prop:alpha-alpha}
Let $\bar{\alpha}$ be given in model~\eqref{model:square-new} and let $\alpha^*$ be given in \eqref{def:alpha-star}. Assume that the set $\{\vx: \mA \vx =\vb\}$ is non-empty and $\mathrm{Ker}(\mA) \neq \{\mathbf{0}\}$. Then the following statements hold.
\begin{itemize}
\item[(i)] $\bar{\alpha} \le \alpha^*$
\item[(ii)] $\bar{\alpha} = \alpha^*$ if and only if there exists an unbounded minimizing sequence of \eqref{model:square-new}.
\end{itemize}
\end{proposition}
\begin{proof}
Item (i). For any $\vx \in \{\vx: \mA \vx =\vb\}$ and nonzero vector $\vd \in \mathrm{Ker}(\mA)$, since $\mA(\vx+t\vd)=\vb$, we have
$$
\bar{\alpha} \le \frac{\|\vx+t\vd\|_1^2}{\|\vx+t\vd\|_2^2}=\frac{\|\vx/t+\vd\|_1^2}{\|\vx/t+\vd\|_2^2}
$$
for all $t\in \mathbb{R}$. Letting $t$ approach to infinity for the above inequality leads to
$\bar{\alpha} \le \frac{\|\vd\|_1^2}{\|\vd\|_2^2}$ for all $\vd \in \mathrm{Ker}(\mA)$. Hence, item (i) holds.

Item (ii). Suppose $\bar{\alpha} = \alpha^*$. There exists a unit vector $\vv \in \mathrm{Ker}(\mA)$ such that $\alpha^* = \|\vv\|_1^2$. Let $\vx_0$ be an arbitrary vector such that $\mA \vx_0=\vb$. We have
$$
\lim_{k \rightarrow \infty}\frac{\|\vx_0+k\vv\|_1^2}{\|\vx_0+k\vv\|_2^2} = \|\vv\|_1^2 = \alpha^*=\bar{\alpha}.
$$
Then, $\{\vx_0+k\vv\}_{k=1}^\infty$ is unbounded and is minimizing sequence of \eqref{model:square-new}.

On the contrary, suppose that $\{\vx_k\}_{k=1}^\infty$ is an unbounded minimizing sequence of \eqref{model:square-new}. Without loss of generality, we have $\lim_{k \rightarrow \infty}\|\vx_k\|_2=\infty$ and $\lim_{k \rightarrow \infty}\frac{\vx_k}{\|\vx_k\|_2}=\vx_*$. Then, we have
$\mA \vx_* = \lim_{k \rightarrow \infty}\frac{\mA\vx_k}{\|\vx_k\|_2}=\mathbf{0}$.
Hence, $\alpha^* \le \|\vx_*\|_1^2 =\frac{\|\vx_*\|_1^2}{\|\vx_*\|_2^2} =\lim_{k\rightarrow\infty}\frac{\|\vx_k\|_1^2}{\|\vx_k\|_2^2}=\bar{\alpha}$.
By item (i), we conclude $\bar{\alpha} = \alpha^*$.
\end{proof}

Prior to examining the behavior of function $F$ in \eqref{model-Trad-FP-square} based on $\alpha^*$, and understanding the relationship between $\bar{\alpha}$ in \eqref{model:square-new} and $\alpha^*$, one might initially consider implementing a bisection-search algorithm on the interval $[1, n]$ to locate the root of $F(\alpha)$, drawing upon from \cite[Proposition 2.1]{Crouzeix-Ferland-Schaible:JOTA-1985}. However, one can see for values of $\alpha^*$ in the range $1 < \alpha^* < n$, the function $F$ lacks the continuity necessary for effectively deploying bisection-search. Beyond this, the inherent nonconvexity of \eqref{model-Trad-FP-square} minimization potentially renders a bisection-search-based algorithm inefficient and, more critically, prone to converging to suboptimal solutions. These considerations highlight the importance of careful algorithmic design when solving the $\tau_2$-model using Dinkelbach’s procedure, as guided by the aforementioned Propositions~\ref{prop:F(alpha)} and \ref{prop:alpha-alpha}.

\subsection{Solutions Existence for Problem~\eqref{model:square-new}}  
In this subsection, we will discuss the existence of global optimal solutions of Problem~\eqref{model:square-new} based on the concept of spherical section property.

\begin{definition} (Spherical section property). Let $m, n$ be two positive integers such that $m<n$. Let $V$ be an $(n-m)$-dimensional subspace of $\mathbb{R}^n$ and $s$ be a positive integer. We say that $V$ has the $s$-spherical section property if
\begin{equation*}
    \inf_{\vv\in V\backslash\{\mathbf{0}\}}\frac{\|\vv\|_1^2}{\|\vv\|_2^2}\ge {\frac{m}{s}}.
\end{equation*}
\end{definition}

It was pointed in \cite{Zhang:JORSC:2013} that if $\mA \in \mathbb{R}^{m \times n}$ (where $m<n$) is a random matrix with independent and identically distributed (i.i.d.) standard Gaussian entries, then its $(n-m)$-dimensional null space exhibits the $s$-spherical section property for $s=c_1(\log(n/m)+1)$ with a probability of at least $1-e^{-c_0(n-m)}$. Here, $c_0$ and $c_1$ are positive constants that remain independent of $m$ and $n$.

With the concept of the spherical section property, we establish the existence of optimal solutions to the model~\eqref{model:square-new} under suitable assumptions.

\begin{theorem} \label{thm:existence}
For the matrix $\mA$ in  model~\eqref{model:square-new}, suppose that $\mathrm{Ker}(\mA)$ has the $s$-spherical section property for some $s>0$ and there exists a vector $\tilde{\vx}\in \mathbb{R}^n$ such that $\|\tilde{\vx}\|_0<\frac{m}{s}$ and $\mA\tilde{\vx}=\vb$. Then, the set of optimal solutions of \eqref{model:square-new} is nonempty.
\end{theorem}
\begin{proof}\ \ We are aware that $\alpha^* \ge m/s$ due to the definition of $s$-spherical section property of $\mathrm{Ker}(\mA)$, and $\|\tilde{\vx}\|_1^2/\|\tilde{\vx}\|_2^2 \le \|\tilde{\vx}\|_0^2$ by virtue of the Cauchy-Schwarz inequality. We conclude that
$$
\bar{\alpha} \le \frac{\|\tilde{\vx}\|_1^2}{\|\tilde{\vx}\|_2^2}<\frac{m}{s} \le \alpha^*.
$$
By Proposition~\ref{prop:alpha-alpha}, there exists a bounded minimizing sequence $\{\vx_k\}$ for~\eqref{model:square-new}. We can select a convergent subsequence $\{\vx_{k_j}\}$ of $\{\vx_k\}$ so that $\lim_{j \rightarrow \infty}\vx_{k_j} = \vx_*$ satisfies $\mA \vx_*=\vb$. We see that
$$
\frac{\|{\vx}_*\|_1^2}{\|\vx_*\|_2^2}=\lim_{j \rightarrow \infty}\frac{\|{\vx}_{k_j}\|_1^2}{\|\vx_{k_j}\|_2^2}=\bar{\alpha}.
$$
This shows $\vx_*$ is an optimal solution of \eqref{model:square-new}. This completes the proof.
\end{proof}

We note that the above proof essentially mirrors the one provided in \cite[Theorem 3.4]{Zeng-Yi-Pong:SIAMOP:2021}.

\section{The $\tau_2$-Model Can Be Viewed as a Quadratic Programming Problem}\label{sec:QP}

In this section, we will demonstrate that the associated optimization problem from Dinkelbach's procedure for the $\tau_2$-model can be framed as a quadratic programming problem. Specifically, the problem is an LCQP problem when $\epsilon=0$, and a QCQP problem when $\epsilon>0$. 

To achieve this, we rewrite the $\tau_2$-model~\eqref{model:square} as
\begin{equation*}\label{model:square-epsilon}
\bar{\alpha}=\inf\left\{\frac{\|\vx\|_1^2}{\|\vx\|_2^2}: \|\mA \vx- \vb\|_2\le \epsilon, \; \vx\in \mathbb{R}^n\right\}. \tag{$P_\epsilon$}
\end{equation*}
We then define
\begin{equation}\label{model-Trad-FP-square-epsilon}
F(\alpha) = \inf\left\{\|\vx\|_1^2-\alpha \|\vx\|_2^2: \|\mA \vx- \vb\|_2\le \epsilon, \; \vx\in \mathbb{R}^n \right\}, \tag{$Q_\epsilon$}
\end{equation}
which is the Dinkelbach-induced function of \eqref{model:square-epsilon}. For the remainder of the paper, we use the terms $\tau_2$-model and Problem~\eqref{model:square-epsilon} interchangeably.

Assuming the existence of optimal solutions to Problem~\eqref{model:square-epsilon}, solving this problem is equivalent to identifying a numerical value at which the optimal value of Problem \eqref{model-Trad-FP-square-epsilon} becomes zero. Consequently, a deeper understanding of Problem \eqref{model-Trad-FP-square-epsilon} is crucial. In this section, we posit that  Problem \eqref{model-Trad-FP-square-epsilon} effectively constitutes a quadratic programming problem. As a result, existing algorithms and theories related to quadratic programming can be seamlessly applied to address Problem \eqref{model-Trad-FP-square-epsilon}. 

To this end, for $\vx\in \mathbb{R}^n$, we define
$$
\vx_{+}=\max\{\vx,\bm{0}\} \quad \mbox{and} \quad \vx_{-}=-\min\{\vx,\bm{0}\}.
$$
Here, the vector $\vx_{+}$ precisely captures the positive entries of $\vx$, while setting the remaining of $\vx$ to zero. Similarly, the vector $\vx_{-}$ precisely records the absolute values of the negative entries of $\vx$, while setting the remaining values of $\vx$ to zero. It is evident that  both $\vx_{+}$ and $\vx_{-}$ belong to $\mathbb{R}^n$ and are non-negative. We then express $\vx$ as the difference between $\vx_{+}$ and $\vx_{-}$, that is,  $\vx=\vx_{+}-\vx_{-}$. Furthermore, we write
$\vv=\begin{bmatrix}\vx_{+} \\ \vx_{-}\end{bmatrix}.$

\subsection{An Indefinite Quadratic Form of the Objective Function for Problem \eqref{model-Trad-FP-square-epsilon}}

Note that 
 $$
\|\vx\|_1 = \|\vx_{+}\|_1+\|\vx_{-}\|_1=\begin{bmatrix}   \vx_{+}^\top &\vx_{-}^\top  \end{bmatrix}\begin{bmatrix}   \ve \\ \ve \end{bmatrix} 
$$
and
$$
\|\vx\|_2^2 = \begin{bmatrix} \vx_{+}^\top &\vx_{-}^\top  \end{bmatrix} \begin{bmatrix} \mI&-\mI \\ -\mI&\mI \end{bmatrix}  \begin{bmatrix} \vx_{+} \\ \vx_{-}  \end{bmatrix}.
$$
Here, $\ve$ is the vector of dimensions $n$ with all entries equal to 1, and $\mI$ is the identity matrix $n \times n$. With these preparations, the objective function of \eqref{model-Trad-FP-square-epsilon} can be written as a quadratic form in terms of $\vx_{+}$ and $\vx_{-}$ as follows
$$
\|\vx\|_1^2 -\alpha \|\vx\|_2^2 =
\begin{bmatrix}\vx_{+}^\top & \vx_{-}^\top\end{bmatrix}
\begin{bmatrix}\ve\ve^\top-\alpha \mI &  \ve\ve^\top+\alpha \mI\\
\ve\ve^\top+\alpha \mI & \ve\ve^\top-\alpha \mI\end{bmatrix}\begin{bmatrix}\vx_{+} \\ \vx_{-}\end{bmatrix}.
$$
Write
\begin{equation}\label{eq:H}
\mH = \begin{bmatrix}\ve\ve^\top-\alpha \mI &  \ve\ve^\top+\alpha \mI\\
\ve\ve^\top+\alpha \mI & \ve\ve^\top-\alpha \mI\end{bmatrix}. 
\end{equation}
Thus, the objective function of Problem \eqref{model-Trad-FP-square-epsilon} can be expressed as 
\begin{equation}\label{model:original-QP-0}
f(\vv):=\vv^\top  \mH\  \vv,
\end{equation}
which represents a quadratic form in terms of the vector $\vv$.

To demonstrate the indefiniteness of this quadratic form, let $\mD$ be the $n\times n$ discrete cosine transform matrix of second kind, which the $(i,j)$th entry of $\mD$ is given by
$$
\sqrt{\frac{2-\delta_{1i}}{n}} \cos \left(\frac{(i-1)(2j-1)\pi}{2n}\right), \quad 1\le i,j \le n,
$$
where $\delta_{ij}$ is the Kronecker delta. We note that $\mD$ is an orthogonal matrix, i.e., $\mD \mD^\top=\mI$. 

\begin{proposition}\label{prop:H-eigendecomp}
The matrix $\mH$ given in \eqref{eq:H} has eigenvalues $2n$, $-2\alpha$, $0$ with multiplicity $1$, $n$, and $n-1$, respectively.  The corresponding eigenvectors are from the columns of $\mathrm{diag}(\mD^\top, \mD^\top)\mE$, where $\mD$ is the discrete cosine transform matrix of size $n \times n$ and
$$
\mE=\frac{\sqrt{2}}{2}\begin{bmatrix}\mE_{11}&\mE_{11} \\ \mE_{11}&-\mE_{11}\end{bmatrix}  + \frac{\sqrt{2}}{2}\sum_{k=2}^n \begin{bmatrix}\mE_{kk}& \mE_{kk} \\ -\mE_{kk}&\mE_{kk}\end{bmatrix}
$$
with $\mE_{ij}$ being the $n\times n$ matrix having a single  nonzero entry $1$ at its $i$th row and $j$th column.
\end{proposition}

\begin{proof}\ \ Since $\ve \ve^\top$ is a rank-one matrix, we know that 
$$
\mD \ve \ve^\top \mD^\top = \Gamma,
$$
where $\Gamma$ is a diagonal matrix whose first diagonal element is $n$ and the rest are zero.  As a consequence, we have $\mD^\top (\ve \ve^\top \pm \alpha \mI)\mD = \Gamma \pm \alpha \mI$ which further implies 
$$
\begin{bmatrix}\mD&\\ &\mD\end{bmatrix} \mH \begin{bmatrix}\mD^\top&\\ &\mD^\top\end{bmatrix}
=\begin{bmatrix}
\Gamma-\alpha \mI&\Gamma+\alpha \mI\\
\Gamma+\alpha \mI&\Gamma-\alpha \mI
\end{bmatrix}.
$$
The matrix on the right hand side, after row and column permutations, is similar to the block diagonal matrix with one $\begin{bmatrix}n-\alpha&n+\alpha\\ n+\alpha&n-\alpha\end{bmatrix}$ block and the other $n-1$ identical blocks of $\begin{bmatrix}-\alpha&\alpha \\ \alpha &-\alpha\end{bmatrix}$ on its diagonal.  Note that
$$
\begin{bmatrix}n-\alpha&n+\alpha\\ n+\alpha&n-\alpha\end{bmatrix} = \begin{bmatrix}\frac{\sqrt{2}}{2}&\frac{\sqrt{2}}{2} \\ \frac{\sqrt{2}}{2}&-\frac{\sqrt{2}}{2}\end{bmatrix}\begin{bmatrix}2n&\\ &-2\alpha\end{bmatrix} \begin{bmatrix}\frac{\sqrt{2}}{2}&\frac{\sqrt{2}}{2} \\ \frac{\sqrt{2}}{2}&-\frac{\sqrt{2}}{2}\end{bmatrix}
$$
and
$$
\begin{bmatrix}-\alpha&\alpha\\ \alpha&-\alpha\end{bmatrix} = \begin{bmatrix}\frac{\sqrt{2}}{2}&\frac{\sqrt{2}}{2} \\ -\frac{\sqrt{2}}{2}&\frac{\sqrt{2}}{2}\end{bmatrix}\begin{bmatrix}-2\alpha&\\ & 0\end{bmatrix} \begin{bmatrix}\frac{\sqrt{2}}{2}& -\frac{\sqrt{2}}{2} \\ \frac{\sqrt{2}}{2}&\frac{\sqrt{2}}{2}\end{bmatrix}.
$$
Hence,
$$
 \mH =\begin{bmatrix}\mD^\top&\\ &\mD^\top\end{bmatrix} \mE^\top\cdot
\mathrm{diag}\left(2n, \underbrace{-2\alpha, \ldots, -2\alpha}_{\mbox{$n$ terms}}, \underbrace{0,\ldots, 0}_{\mbox{$(n-1)$ terms}}\right) \cdot \mE \begin{bmatrix}\mD&\\ &\mD\end{bmatrix}.
$$
This completes the proof.
\end{proof}

According to Proposition~\ref{prop:H-eigendecomp}, the quadratic form $f$ in equation \eqref{model:original-QP-0} is indefinite when $\alpha \neq 0$ and semi-positive definite otherwise. 

\subsection{The LCQP or QCQP Reformulation  for Problem \eqref{model-Trad-FP-square-epsilon}}

In the previous subsection, we demonstrated that the objective function of Problem \eqref{model-Trad-FP-square-epsilon} is expressed as a quadratic function. In this subsection, we identify  Problem \eqref{model-Trad-FP-square-epsilon} as either an LCQP or a QCQP problem. 

To this end, we examine  the feasible set of Problem \eqref{model-Trad-FP-square-epsilon} in terms of the vector $\vv$, given by
$$
\Omega_\epsilon=\{\vv \in \mathbb{R}^{2n}: \left\|\begin{bmatrix}\mA & -\mA\end{bmatrix} \vv - \vb\right\|_2 \le \epsilon \;\mbox{and} \; \vv \ge \vzero\}.
$$  
We consider two cases: $\epsilon=0$ and $\epsilon>0$. 

Case 1: $\epsilon=0$.  In this case, the feasible set simplifies to  
$$
\Omega_\epsilon=\{\vv \in \mathbb{R}^{2n}: \begin{bmatrix}\mA & -\mA\end{bmatrix} \vv = \vb, \; \mbox{and} \; \vv \ge \vzero\}.
$$  
Let $\mathcal{E}=\{1,2,\ldots,m\}$, $b_i$ be the $i$th element of $\vb$, and the vector $\vc_i$ be the $i$th column of $\begin{bmatrix}\mA & -\mA\end{bmatrix}^\top$ for $i \in \mathcal{E}$. Then, the optimization problem~\eqref{model-Trad-FP-square-epsilon} can be written as
\begin{eqnarray}
\mbox{minimize} && f(\vv) = \vv^\top  \mH\  \vv  \label{model:original-QP-1} \\
\mbox{subject to} && \langle \vc_i, \vv\rangle = b_i, \quad \quad i \in \mathcal{E}   \label{model:original-QP-2} \\
&&\vv \ge \vzero.   \label{model:original-QP-3}
\end{eqnarray}
This represents a typical LCQP problem.

Case 2:  $\epsilon>0$. Here, the feasible set becomes  
$$
\Omega_\epsilon=\{\vv \in \mathbb{R}^{2n}: \vv^\top \begin{bmatrix}\mA^\top \mA&-\mA^\top \mA\\ -\mA^\top \mA &\mA^\top \mA \end{bmatrix}\vv-2\left\langle\begin{bmatrix}\mA^\top \vb \\ -\mA^\top \vb\end{bmatrix}, \vv \right\rangle+(\|\vb\|_2^2-\epsilon^2) \le 0 \;\mbox{and} \; \vv \ge \vzero\}.
$$ 
Consequently, the optimization problem~\eqref{model-Trad-FP-square-epsilon} can be reformulated as
\begin{eqnarray}
\mbox{minimize}&& f(\vv) = \vv^\top  \mH\  \vv  \label{model:original-QP-1-epsilon} \\
\mbox{subject to} && \vv^\top \begin{bmatrix}\mA^\top \mA&-\mA^\top \mA\\ -\mA^\top \mA &\mA^\top \mA \end{bmatrix}\vv-2\left\langle\begin{bmatrix}\mA^\top \vb \\ -\mA^\top \vb\end{bmatrix}, \vv \right\rangle+(\|\vb\|_2^2-\epsilon^2) \le 0  \label{model:original-QP-2-epsilon} \\
&&\vv \ge \vzero.   \label{model:original-QP-3-epsilon}
\end{eqnarray}
This represents a typical QCQP problem.

Both the LCQP and the QCQP formulations presented above are indefinite, making them challenging to solve. To address this challenge, we can convexify the objective function by linearizing the term $\alpha \|\vx\|_2^2$, which is a common practice. The linearization of  $\alpha \|\vx\|_2^2$,  achieved by omitting the constant term, takes the form $2\alpha \langle \va,\vx \rangle$ for some vector $\va$. As a result, Problem~\eqref{model-Trad-FP-square-epsilon} transforms into the following
\begin{equation}\label{model:Linerization}
\argmin_{\vx\in \mathbb{R}^n}\{\|\vx\|_1^2-2\alpha \langle \va,\vx \rangle: \|\mA \vx- \vb\|_2\le \epsilon\}.
\end{equation}
With
$$
\mH=\begin{bmatrix}\ve\ve^\top&\ve\ve^\top\\ \ve\ve^\top&\ve\ve^\top\end{bmatrix} \quad \mbox{and}\quad \vq=\begin{bmatrix}\va \\ -\va \end{bmatrix}, 
$$
the optimization problem~\eqref{model:Linerization} can be reformulated as an LCQP if $\epsilon=0$ or a QCQP if $\epsilon>0$, by simply replacing the objective function in \eqref{model:original-QP-1} and  \eqref{model:original-QP-1-epsilon} with 
$$
\min_{\vv \in \mathbb{R}^{2n}} f(\vv) = \vv^\top  \mH\  \vv +2 \alpha \langle \vq,\vv \rangle.  
$$
By Proposition~\ref{prop:H-eigendecomp}, the matrix $\mH$ has $2n$ as its only nonzero eigenvalue. Hence, the corresponding linearized LCQP and QCQP are typical convex quadratic programming problems.

\section{Algorithms for the $\tau_2$-Model} \label{sec:convergence}

In this section, we develop algorithms to solve the $\tau_2$-model, that is, the optimization problem~\eqref{model:square-epsilon}. We initially employ Dinkelbach’s procedure, leading to an iterative scheme as follows: starting with $\vx^{(0)}$ such that $\mA \vx^{(0)}=\vb$, iterate
\begin{align} \label{eqn:scheme} \tag{${Q}^k_\epsilon$}
\begin{cases}
  \vx^{(k+1)}&=\arg \inf \{\|\vx\|_1^2 - \alpha^{(k)} \|\vx\|_2^2: \|\mA\vx-\vb\|_2 \le \epsilon\};\\
  \alpha^{(k+1)}&=\displaystyle\frac{\|\vx^{(k+1)}\|_1^2}{\|\vx^{(k+1)}\|_2^2}.
\end{cases}
\end{align}
The optimization problem~\eqref{eqn:scheme} is nonconvex, as indicated by Proposition~\ref{prop:H-eigendecomp}, and may not be solvable, as asserted by Proposition~\ref{prop:alpha-alpha}. To overcome this, as mentioned in the previous section, we propose to lineralize the term  $\|\vx\|_2^2$ in the objective function at the point $\vx^{(k)}$, resulting in the following iterative scheme:   
\begin{align} \label{eqn:scheme-linearize} \tag{${L}^k_\epsilon$}
\begin{cases}
  \vx^{(k+1)}&=\arg \min \{\|\vx\|_1^2 - 2\alpha^{(k)} \langle \vx^{(k)}, \vx \rangle: \|\mA\vx-\vb\|_2 \le \epsilon\};\\
  \alpha^{(k+1)}&=\displaystyle\frac{\|\vx^{(k+1)}\|_1^2}{\|\vx^{(k+1)}\|_2^2}.
\end{cases}
\end{align}

In the following two subsections, we first prove that the sequence $\{\vx^{(k)}\}$ generated by \eqref{eqn:scheme-linearize} converges to a stationary point of problem~\eqref{model:square-epsilon}. Secondly, we propose two different approaches to solve the optimization problem in \eqref{eqn:scheme-linearize}.

\subsection{Convergence}

The iterative scheme in \eqref{eqn:scheme-linearize} generates two sequences $\{\alpha^{(k)}\}$ and $\{\vx^{(k)}\}$. To investigate the convergence of the sequence $\{\vx^{(k)}\}$, we need to recall some concepts and establish supporting lemmas. 

A extended real-valued function $f:\mathbb{R}^n \rightarrow (-\infty, \infty]$ is said to be proper if its domain $\mathrm{dom} f:=\{\vx: f(\vx)<\infty\}$. Additionally, a proper function $f$ is said to be closed if it is lower semi-continuous. For a proper closed function $f$, the regular subdifferential $\hat{\partial} f(\bar{\vx})$  and the limiting subdifferential $\partial f(\bar{\vx})$ at $\bar{\vx}\in \mathrm{dom} f$   are given respectively as 
\begin{eqnarray*}
    \hat{\partial} f(\bar{\vx})&:=&\left\{\vv: \lim_{\vx \rightarrow \bar{\vx}}\inf_{\vx \neq \bar{\vx}} \frac{f(\vx)-f(\Bar{\vx})-\langle \vv, \vx-\Bar{\vx}\rangle}{\|\vx-\bar{\vx}\|_2}\ge 0 \right\},\\
    \partial f(\bar{\vx})&:=&\left\{\vv: \exists  \vx^{(t)} \stackrel{f}{\rightarrow} \bar{\vx}\; \mbox{and}\; \vv^{(t)}\in \hat{\partial}f(\vx^{(t)}) \; \mbox{with} \; \vv^{(t)} \rightarrow \vv \right\}, 
\end{eqnarray*}
where $\vx^{(t)} \stackrel{f}{\rightarrow} \bar{\vx}$ means $\vx^{(t)} \rightarrow \vx$ and $f(\vx^{(t)}) \rightarrow f(\vx)$. For a proper closed function $f$, we state that $\bar{\vx}$ is a stationary point of $f$ and that $0 \in \partial f(\bar{\vx})$. 

For a closed non-empty set $C$, we define the indicator function $\iota_C(\vx)$ as $0$ if $\vx\in C$ and $\infty$ otherwise. 

With these notations, the constrained optimization problem~\eqref{model:square-epsilon} can be reformulated as an unconstrained optimization problem with the following objective function: 
\begin{equation}\label{def:G}
G(\vx)=\frac{\|\vx\|_1^2}{\|\vx\|_2^2} + \iota_{\mathcal{B}_\epsilon(\mathbf{0})}(\mA \vx-\vb),
\end{equation}
where $\mathcal{B}_\epsilon(\vz)$ is the ball centered at $\vz$ with radius $\epsilon$. 
For this function $G$, we know $\mathrm{dom} G:=\{\vx: \|\mA\vx-\vb\|_2 \le \epsilon\}$ and $G$ is continuous on its domain. On the $\mathrm{dom} G$, by the calculus of subdifferentials, it holds that 
\begin{equation}\label{eq:partial-G}
\partial G(\vx) = \left(\frac{2\|\vx\|_1}{\|\vx\|_2^2}\right)\partial \|\cdot\|_1(\vx) -\left(\frac{2\|\vx\|_1^2}{\|\vx\|_2^4}\right) \vx+\mA^\top \partial\iota_{\mathcal{B}_\epsilon(\mathbf{0})}(\mA \vx-\vb).    
\end{equation}
We see from the above relation that $\mathrm{dom} \partial G=\mathrm{dom} G$. 

Next, let us first review the definition of the Kurdyka-\L ojasiewicz (KL) property of a function and recall a convergence theorem on a function having the KL property.

Let $f:\mathbb{R}^d\to (-\infty,+\infty]$ be proper and lower semicontinuous.  We note that $[f<\mu]:=\{\vx\in \mathbb{R}^n: f(\vx)<\mu\}$ and $[\eta<f<\mu]:=\{\vx\in \mathbb{R}^n: \eta<f(\vx)<\mu\}$. Let $r_0>0$ and set 
$$
\mathcal{K}(r_0):=\{\varphi: \varphi \in C^0([0,r_0))\cap C^1((0,r_0)), \; \varphi(0)=0, \; \varphi\; \mbox{is concave and}\; \varphi' >0 \}. 
$$
The function $f$ satisfies the KL inequality (or has KL property) locally  at $\tilde{\vx}\in \dom \partial f$ if there exist $r_0>0$, $\varphi \in \mathcal{K}(r_0)$ and a neighborhood $U(\tilde{\vx})$ of $\tilde{\vx}$ such that
\begin{equation}\label{KL-inequality}
\varphi'(f(\vx)-f(\tilde{\vx}))\text{dist}(\mathbf{0},\partial f(\vx))\geq 1
\end{equation}
for all $\vx \in U(\tilde{\vx}) \cap [f(\tilde{\vx})<f(\vx)<f(\tilde{\vx})+r_0]$. The function $f$ has the KL property on $S$ if it does so at each point of $S$.

Since $\|\cdot\|_1^2$, $\|\cdot\|_2^2$, and $\iota_{\{\mathbf{0}\}}$ are semialgebraic, so is $G$. Therefore, $G$ has the KL property in its domain; see \cite{Attouch-Bolte-Svaiter:MP:13}.

The convergence analysis of the sequence $\{\vx^{(k)}\}$ is motivated by the inexact descent convergence results for KL functions in \cite{Attouch-Bolte-Svaiter:MP:13}. Here are three essential conditions to guarantee the convergence of the sequence $\{\vx^{(k)}\}$  generated by \eqref{eqn:scheme-linearize}.
\begin{enumerate}
\item[\textbf{(H1)}] \textbf{Sufficient descent condition:} There exists a positive constant $c_1$ such that for $\forall k\in \mathbb{N}$,
\begin{displaymath}
c_1 \|\vx^{(k+1)}-\vx^{(k)}\|_2^2\leq G(\vx^{(k)})-G(\vx^{(k+1)}).
\end{displaymath}
\item[\textbf{(H2)}] \textbf{Relative error condition:} There exists a positive constant $c_2$ such that  for $\forall k\in \mathbb{N}$,
\begin{displaymath}
\|\mathbf{\omega}^{(k+1)} \|_2\leq c_2\|\vx^{(k+1)}-\vx^{(k)}\|_2\quad\text{ and } \quad \mathbf{\omega}^{(k+1)}\in \partial G(\vx^{(k+1)}).
\end{displaymath}
\item[\textbf{(H3)}] \textbf{Continuity condition:} There exists a subsequence $\{\vx^{(k_t)}\}_{t\in\mathbb{N}}$ and $\vx^*$ such that
\begin{displaymath}
\lim_{t\to\infty}\vx^{(k_t)}= \vx^*\quad\text{ and }\quad \lim_{t\to\infty}G(\vx^{(k_t)})= G(\vx^*).
\end{displaymath}
\end{enumerate}

Before we present the convergence analysis of the sequence $\{\vx^{(k)}\}$  generated by \eqref{eqn:scheme-linearize}, we require two additional lemmas. The following lemma addresses the monotonic decreasing property of the sequence $\{\alpha^{(k)}\}$.
\begin{lemma}\label{lemma:alpha-decreasing}
The sequence $\{\alpha^{(k)}\}$ generated by \eqref{eqn:scheme-linearize} is decreasing.
\end{lemma}
\begin{proof}\ \ Since the vector $\vx^{(k+1)}$ is the minimizer of the optimization problem in \eqref{eqn:scheme-linearize}, we have
$$
\|\vx^{(k+1)}\|_1^2 - 2\alpha^{(k)} \langle \vx^{(k)}, \vx^{(k+1)} \rangle \le \|\vx^{(k)}\|_1^2 - 2\alpha^{(k)} \|\vx^{(k)}\|_2^2.
$$ 
This leads to:
$$
\alpha^{(k+1)}=\frac{\|\vx^{(k+1)}\|_1^2}{\|\vx^{(k+1)}\|_2^2}\le \frac{2\alpha^{(k)} \langle \vx^{(k)}, \vx^{(k+1)} \rangle + \|\vx^{(k)}\|_1^2 - 2\alpha^{(k)} \|\vx^{(k)}\|_2^2}{\|\vx^{(k+1)}\|_2^2}.
$$
Using the fact that $\|\vx^{(k)}\|_1^2=\alpha^{(k)}\|\vx^{(k)}\|_2^2$, we have 
$$2\alpha^{(k)} \langle \vx^{(k)}, \vx^{(k+1)} \rangle + \|\vx^{(k)}\|_1^2 - 2\alpha^{(k)} \|\vx^{(k)}\|_2^2=\alpha^{(k)}(\|\vx^{(k+1)}\|_2^2-\|\vx^{(k+1)}-\vx^{(k)}\|_2^2).$$ 
Thus, 
$$
\alpha^{(k+1)} \le \alpha^{(k)} \left(1-\frac{\|\vx^{(k+1)}-\vx^{(k)}\|_2^2}{\|\vx^{(k+1)}\|_2^2}\right).
$$
This implies $\alpha^{(k+1)} \le \alpha^{(k)}$, which completes the proof. 
\end{proof}

The next lemma is also necessary in our convergence analysis. 
\begin{lemma} \label{lemma:lipschitz}
Define
$$
\Phi(\vx):=\frac{\|\vx\|_1^2}{\|\vx\|_2^2}\vx.
$$
Then, for any $\vx, \vy\in \mathbb{R}^n$, we have
$$
\|\Phi(\vx)-\Phi(\vy)\| \le 5n \|\vx-\vy\|_2.
$$
\end{lemma}
\begin{proof}
First, we claim that the function $\Phi$ is $L$-Lipschitz continuous if and only if for any $\vx_0\in \mathbb{R}^n$ and every pair of unit vector $\vu$ and $\vv$ of $\mathbb{R}^n$, the real-valued function
\begin{equation*}
    \varphi(t)=\langle \Phi(\vx_0 + t\vu), \vv\rangle  
\end{equation*}
is $L$-Lipschitz continuous. Actually, if $\Phi$ is $L$-Lipschitz continuous, then
$$
|\varphi(s)-\varphi(t)\| \le \left |(\Phi(\vx_0 + s\vu)-\Phi(\vx_0 + t\vu))\cdot \vv \right| \le L |(s-t)|\|\vu\|_2\|\vv\|_2=L |(s-t)|.
$$
On the other hand, for any $\vx_0\in \mathbb{R}^n$ and every pair of unit vector $\vu$ and $\vv$ of $\mathbb{R}^n$, if $\varphi$ is $L$-Lipschitz, then  for any vectors $\vx$ and $\vy$ in $\mathbb{R}^n$, by identifying $\vx_0$ and $\vu$ as $\vx$ and $\frac{\vy-\vx}{\|\vy-\vx\|_2}$, respectively, we have that
$$
\|\Phi(\vx)-\Phi(\vy)\|_2 =\max_{\|\vv\|_2=1} \left\langle \Phi(\vx)-\Phi\left(\vx+\frac{\vy-\vx}{\|\vy-\vx\|_2}\|\vy-\vx\|_2\right), \vu\right \rangle \le L\|\vy-\vx\|_2,
$$
which implies $\Phi$ being $L$-Lipschitz continuous.

Next, we focus on showing that $\varphi$ is a Lipschitz continuous function.  For a fixed $\vx_0$ such that $\|\mA \vx_0-\vb\|_2\le \epsilon$, a unit vector $\vu \in \mathrm{Ker} (\mA)$, and a unit vector $\vv$ of $\mathbb{R}^n$, set 
$$f(t)=\|\vx_0+t\vu\|_1, \; g(t)=\langle \vx_0+t\vu, \vv \rangle, \; \mbox{and}\; h(t)=\|\vx_0+t\vu\|_2.
$$
We can check directly that $f$, $g$, and $h$ are Lipschitz continuous with constants $\sqrt{n}$, $1$, and $1$, respectively. Furthermore, both $g$ and $h$ are differentiable, and $f$ is differentiable almost everywhere.

We further define
$$
\varphi(t)=\frac{f(t)^2 g(t)}{h(t)^2}.
$$
From
$$
\varphi'(t)=\frac{2f(t)f'(t)g(t)}{h^2(t)}+\frac{f^2(t) g'(t)}{h^2(t)}-2\frac{f^2(t) g(t) h'(t)}{h^3(t)}
$$
and $f(t) \le \sqrt{n} h(t)$ and $|g(t)|\le h(t)$, we obtain
$$
\sup|\varphi'(t)|\le \sup \frac{2f(t)|f'(t)|g(t)}{h^2(t)}+\frac{f^2(t) |g'(t)|}{h^2(t)}+2\frac{f^2(t) g(t) |h'(t)|}{h^3(t)}\le 5n.
$$
By \cite[Exercise 3.37]{folland1999real}, $\varphi$ is $L$-Lipschitz continuous with $L=5n$, so is the function $\Phi$.
\end{proof}

Now, we are ready to present the convergence of analysis for the sequence $\{\vx^{(k)}\}$ generated by \eqref{eqn:scheme-linearize}.
\begin{theorem}\label{thm:convergence}
Given a sequence $\{\vx^{(k)}, \alpha^{(k)}\}$ generated by \eqref{eqn:scheme-linearize}. If $\{\vx^{(k)}\}$ is bounded, then $\{\vx^{(k)}\}$ converges to a stationary point of $G$ in \eqref{def:G}.
\end{theorem}

\begin{proof}\ \ Since $\{\vx^{(k)}\}$ is bounded and $\|\mA \vx^{(k)} -\vb\|_2\le \epsilon$, we know that $M_1 \le \|\vx^{(k)}\|_2 \le M_2$ for two numbers with $0<M_1\le M_2 <\infty$. From the proof of Lemma~\ref{lemma:alpha-decreasing}, we have
$$
\frac{\alpha^{(k)}}{\|\vx^{(k+1)}\|_2^2}\|\vx^{(k+1)}-\vx^{(k)}\|_2^2 \le \alpha^{(k)} -\alpha^{(k+1)}=G(\vx^{(k)})-G(\vx^{(k+1)}).
$$
Hence, condition \textbf{H1} holds with $c_1=M_2^{-2}$.

Next, we show that the sequence $\{\vx^{(k)}\}$ satisfies condition \textbf{H2}. By Fermat's rule for the convex optimization problem~\eqref{eqn:scheme-linearize}, there exists $\vs^{(k+1)} \in \partial \iota_{\mathcal{B}_\epsilon(\mathbf{0})}(\mA \vx^{(k+1)}-\vb)$ such that
$$
\mathbf{0} \in \partial \|\cdot\|_1^2(\vx^{(k+1)})-2\alpha^{(k)}\vx^{(k)}+\mA^\top \vs^{(k+1)}.
$$ 
Due to $\alpha^{(k)}=\frac{\|\vx^{(k)}\|_1^2}{\|\vx^{(k)}\|_2^2}$, one immediately has
\begin{equation}\label{thm:tmp1}
2\frac{\|\vx^{(k)}\|_1^2}{\|\vx^{(k)}\|_2^2}\vx^{(k)} \in \partial \|\cdot\|_1^2(\vx^{(k+1)})+\mA^\top \vs^{(k+1)}.
\end{equation}
By \eqref{eq:partial-G}, we have that
\begin{eqnarray} \nonumber
&&\frac{1}{\|\vx^{(k+1)}\|_2^2}\partial\|\cdot\|_1^2(\vx^{(k+1)})-2\frac{\|\vx^{(k+1)}\|_1^2}{\|\vx^{(k+1)}\|_2^4}\vx^{(k+1)} + \frac{\mA^\top \vs^{(k+1)}}{\|\vx^{(k+1)}\|_2^2} \\  \label{thm:tmp2}
&=&\partial\left(\frac{\|\cdot\|_1^2}{\|\cdot\|_2^2}\right)(\vx^{(k+1)})+ \frac{\mA^\top \vs^{(k+1)}}{\|\vx^{(k+1)}\|_2^2} \subset \partial G(\vx^{(k+1)}).
\end{eqnarray}
Combining~\eqref{thm:tmp1} and \eqref{thm:tmp2}, we have that
$$
\frac{2}{\|\vx^{(k+1)}\|_2^2}\left(\frac{\|\vx^{(k)}\|_1^2}{\|\vx^{(k)} \|_2^2}\vx^{(k)}-\frac{\|\vx^{(k+1)}\|_1^2}{\|\vx^{(k+1)}\|_2^2}\vx^{(k+1)}\right) \in \partial G(\vx^{(k+1)}).
$$
By Lemma~\ref{lemma:lipschitz} and the fact of that the norms of $\vx^{(k)}$ is bounded below by a positive number, we conclude
$$
\left\|\frac{2}{\|\vx^{(k+1)}\|_2^2}\left(\frac{\|\vx^{(k)}\|_1^2}{\|\vx^{(k)} \|_2^2}\vx^{(k)}-\frac{\|\vx^{(k+1)}\|_1^2}{\|\vx^{(k+1)}\|_2^2}\vx^{(k+1)}\right) \right\|_2 \le c \|\vx^{(k+1)}-\vx^{(k)}\|_2
$$
for some positive number $c$. Hence, the condition \textbf{H2} holds.

Since $\{\vx^{(k)}\}$ is bounded, there exists a subsequence $\{\vx^{(k_j)}\}$ and $\bar \vx$ such that $\vx^{(k_j)} \to \bar\vx$ as $j\to \infty$. Since $G$ is continuous on its domain, we conclude that $G(\vx^{(k_j)}) \to  G(\bar\vx)$ as $j\to \infty$. Hence, the condition \textbf{H3} holds.

Finally, since $G$ is a function having the KL property, then by \cite[Theorem 2.9]{Attouch-Bolte-Svaiter:MP:13} we conclude that $\{\vx^{(k)}\}$ converges to a stationary point of $G$.
\end{proof}

\subsection{Algorithms}

In this subsection, we aim to complete the iterative scheme of \eqref{eqn:scheme-linearize} by presenting two distinct approaches to solving the optimization problem in \eqref{eqn:scheme-linearize}.

\emph{The Quadratic Programming (QP) Approach:} The optimization problem in \eqref{eqn:scheme-linearize}, without the index $k$, is formulated as: 
\begin{equation}\label{eq:opt-sub}
\min \{\|\vx\|_1^2 - 2\alpha \langle \vc, \vx \rangle: \|\mA\vx-\vb\|_2\le \epsilon\}.
\end{equation}
This problem closely resembles the one presented in \eqref{model:Linerization}, which was previously reformulated as a typical LCQP if $\epsilon=0$ or QCQP if $\epsilon>0$,  described in the prevoius section. Given the abundance of well-established algorithms for quadratic programming, we opt for a straightforward selection to address problem~\eqref{eq:opt-sub}. For $\epsilon=0$, we employ the MATLAB function ``\texttt{quadprog}''. For the case of $\epsilon>0$, we use the Gurobi QCQP solver (https://www.gurobi.com/).

\emph{Alternating Direction Linearized Proximal Method of Multipliers (AD-LPMM) Approach:} To present this approach, we rewrite \eqref{eq:opt-sub} into an equivalent form as 
\begin{equation}\label{eq:opt-sub-equi}
\min \{\theta_1(\vx)+\theta_2(\mA \vx): \vx \in \mathbb{R}^n\},
\end{equation}
where 
$$
\theta_1(\cdot)=\|\cdot\|_1^2 - 2\alpha \langle \vc, \cdot \rangle, \quad \theta_2(\cdot)=\iota_{\mathcal{B}_\epsilon(\vb)}.
$$
Both $\theta_1$ and $\theta_2$ are convex; therefore, there are numerical methods that can be used for the problem~\eqref{eq:opt-sub-equi}, for example, see \cite{Chambolle-Pock:JMIV11,Li-Shen-Xu-Zhang:AiCM:15,Zhang-Burger-Osher:JSC:2011}. Here, we simply choose the alternating direction linearized proximal method of multipliers (AD-LPMM) approach presented in \cite[Chapter 15]{Beck:2017} and also see \cite{Chen-Shen-Suter-Xu:Eurasip:15}, that is, for given $\vx^{(0)} \in \mathbb{R}^n$,  $\vy^{(0)},\vz^{(0)} \in \mathbb{R}^m$, $\rho>0$, $\eta \ge \rho \lambda_{\max}(\mA^\top \mA)$, $\beta \ge \rho$, iterate 
\begin{align} \label{eq:AD-LPMM-epsilon}
\begin{cases}
\vx^{(j+1)}=\mathrm{prox}_{\frac{1}{\eta}\theta_1}\left(\vx^{(j)}-\frac{\rho}{\eta}\mA^\top\left(\mA\vx^{(j)}-\vz^{(j)}+\frac{1}{\rho}\vy^{(j)}\right)\right),   \\
\vz^{(j+1)}=\mathrm{prox}_{\frac{1}{\beta}\theta_2}\left(\vz^{(j)}+\frac{\rho}{\beta}\left(\mA\vx^{(j+1)}-\vz^{(j)}+\frac{1}{\rho}\vy^{(j)}\right)\right),   \\
\vy^{(j+1)}= \vy^{(j+1)}+\rho(\mA \vx^{(j+1)}-\vz^{(j+1)}).
\end{cases}
\end{align}
Here, for a proper convex function $f: \mathbb{R}^d \rightarrow \mathbb{R}$, $\mathrm{prox}_{f}(\vu)$ is the proximity operator of $f$ at the point $\vu\in \mathbb{R}^d$ defined as 
$$
\mathrm{prox}_{f}(\vu):=\arg\min\left\{f(\vv)+\frac{1}{2}\|\vv-\vu\|_2^2: \vv\in \mathbb{R}^d \right\}.
$$ 
Clearly, 
$$
\mathrm{prox}_{\frac{1}{\eta} \theta_1}(\vu)=\mathrm{prox}_{\frac{1}{\eta} \|\cdot\|_1^2}\left(\vu+\frac{2\alpha}{\eta}  \vc\right)
$$
and
$$
\mathrm{prox}_{\frac{1}{\beta} \theta_2}(\vu)=\vb + \min\left\{1, \frac{\epsilon}{\|\vu-\vb\|_2}\right\}(\vu-\vb).
$$
To adapt the AD-LPMM scheme to \eqref{eq:opt-sub}, we set $L=\lambda_{\max}(\mA^\top \mA)$ and select $\eta=\rho L$. Specifically, when $\epsilon=0$,   AD-LPMM scheme for \eqref{eq:opt-sub} becomes 
\begin{align} \label{eq:AD-LPMM}
\begin{cases}
        \vx^{(j+1)}=\mathrm{prox}_{\frac{1}{\rho L}\|\cdot\|_1^2}\left(\vx^{(j)}-\frac{1}{L}\mA^\top\left(\mA\vx^{(j)}-\vb+\frac{1}{\rho}\vy^{(j)}\right)+\frac{2\alpha}{\rho L} \vc\right), \\
        \vy^{(j+1)}= \vy^{(j+1)}+\rho(\mA \vx^{(j+1)}-\vb).
\end{cases}
\end{align}

Now, we are ready to present a complete algorithm for Problem~\eqref{model:square-epsilon}. Since this algorithm is essentially based on Dinkelbach's procedure, we refer to Algorithm~\ref{alg:ADLPMM} as {D-QP} if the step (a) in Algorithm~\ref{alg:ADLPMM} is performed with QP; otherwise, it is denoted as {D-LPMM}. 

\begin{algorithm}[H]
\label{alg:ADLPMM}
\caption{A complete algorithm for Problem~\eqref{model:square-epsilon}}

\SetKw{KwIni}{Initialization:}

\KwIn{A matrix $\mA \in \mathbb{R}^{m\times n}$ and a nonzero vector $\vb \in \mathbb{R}^m$}

\KwIni{choose $\vx^{(0)}$ with $\|\mA \vx-\vb\|_2\le \epsilon$, $\alpha^{(0)}=\frac{\|\vx^{(0)}\|_1^2}{\|\vx^{(0)}\|_2^2}$, $L=\lambda_{\max}(\mA^\top \mA)$, $\rho>0$}

\For{$k=0,1,2, \ldots$}{
\begin{itemize}
\item[(a)] finding $\vx^{(k+1)}$ through solving \eqref{eq:opt-sub} with $\vc=\vx^{(k)}$ by either QP or  AD-LPMM through \eqref{eq:AD-LPMM-epsilon};\\
\item[(b)] updating $\alpha^{(k+1)}=\frac{\|\vx^{(k+1)}\|_1^2}{\|\vx^{(k+1)}\|_2^2}$.
\end{itemize}
}

\KwOut{$\vx^{(k)}$}
\end{algorithm}

As the proximity operator $\mathrm{prox}_{\beta\|\cdot\|_1^2}$ is involved in \eqref{eq:AD-LPMM}, to conclude this section, we present an explicit method to compute this proximity operator using Algorithm~\ref{alg:positive}, which was developed in our recent work \cite{Prater-Shen-Tripp:ACHA:23}.

\begin{algorithm}[H]
\label{alg:positive}
\caption{A routine for computing $\mathrm{prox}_{\beta \|\cdot\|_1^2}(x)$ for a vector $\vx \in \mathbb{R}^n$.}

\SetAlgoVlined

\SetKw{KwIni}{Initialization:}

\KwIn{$\beta>0$ and a nonzero vector $\vx \in \mathbb{R}^n$}
\KwIni Rearranging the entries of $\vx$ according to its absolute values in the decreasing order through a signed permutation matrix $\mP$. Set $k=1$ and $r=\frac{x_1}{2\beta+1}$

\While{$k < n$}{
\eIf{$x_{k+1} \le 2\beta  r$}{
$$
u_i=\left\{
      \begin{array}{ll}
        x_i-2\beta r, & \hbox{for $i=1,\ldots,k$;} \\
        0, & \hbox{for $i=k+1,\ldots,n$.}
      \end{array}
    \right.
$$
break \;
}{
 update $k \leftarrow k+1$ and $r= \frac{\sum_{i=1}^k x_i}{2k\beta+1}$\;
}
}
\KwOut{$\mathrm{prox}_{\beta h_p}(\vx) \gets \mP^{-1}\vu$\;}
\end{algorithm}

\section{Numerical experiments}\label{sec:numerical}
In this section, we present numerical experiments to demonstrate the performance of our proposed algorithms, D-QP and D-LPMM, for sparse signal recovery using the $\tau_2$-model.    These experiments include a comparative analysis against recovery algorithms tailored for the $\ell_1/\ell_2$-model, under both noiseless and noisy measurement conditions. All numerical experiments are executed on a desktop equipped with an Intel i7-7700 CPU (4.20GHz) running MATLAB 9.8 (R2020a).

The performance of the tested algorithms will be evaluated using two distinct sensing matrices:

\smallskip
\noindent \emph{Oversampled discrete cosine transform (DCT) matrix:} We define the $m \times n$ sensing matrix $\mA$, where the $j$th column is given by:  
        \begin{equation*}
            \va_j := \frac{1}{\sqrt{m}}\cos\bigg(\frac{2\pi j \vw}{E}\bigg).
        \end{equation*}
Here, $\vw$ is a random vector following uniform distribution in $[0,1]^m$, and $E$ is a positive parameter controlling the coherence. A larger $E$ results in a more coherent matrix. This matrix is frequently used in various applications \cite{fannjiang2012coherence,Rahimi-Wang-Dong-Lou:SIAMSC:2019,Yin-Lou-He-Xin:SIAMSC:2015}, particularly in scenarios where standard $\ell_1$ models struggle due to high coherence. 

\smallskip
\noindent \emph{Gaussian matrix:} Here, the sensing matrix $\mA$ is generated based on multivariate normal distribution $\mathcal{N}(\mathbf{0}, \Sigma)$. For a number $r$ in the range of  $[0,1]$, the $(i,j)$th entry of the covariance matrix $\Sigma$ is specified as
$$
\Sigma_{ij}=\left\{
  \begin{array}{ll}
    1, & \hbox{if $i=j$;} \\
    r, & \hbox{otherwise.}
  \end{array}
\right.
$$
Here, a larger $r$ value indicates a more challenging problem in sparse recovery \cite{zhang2018minimization,Rahimi-Wang-Dong-Lou:SIAMSC:2019}.   In our experiments, the size of the sensing matrices $\mA$ is set to $64 \times 1024$.

The ground truth $\vx\in \mathbb{R}^n$ is simulated as an $s$-sparse signal, where $s$ is the number of nonzero entries in $\vx$. Following the suggestion in \cite{fannjiang2012coherence}, we ensure that the indices of nonzero entries are separated by at least $2E$. Specifically,  to achieve effective sparse recovery with matrices such as oversampled DCT, it is crucial that the nonzero elements of $\vx$ are adequately spaced. This required spacing, termed minimum separation and quantified as Rayleigh length (RL), was explored in \cite{candes2013super}. For our oversampled DCT matrices, the RL is designated as $E$. The empirical research by \cite{fannjiang2012coherence} establishes that a minimum separation of 2RL is essential for optimal sparse recovery. Intuitively, sparse spikes need to be further apart for more coherent matrices.

The magnitudes of nonzero entries are set differently according to the sensing matrix $\mA$. For $\mA$ being an oversampled DCT matrix, the dynamic range of a signal $\vx$ is defined as 
$$
\Theta(\vx)=\frac{\max\{|x_i|: i \in \mathrm{supp} (\vx)\}}{\min\{|x_i|: i \in \mathrm{supp} (\vx)\}},
$$
which can be controlled by an exponential factor $D$. In particular, a MATLAB command for generating those nonzero entries is 
$$\texttt{xg=sign(randn(s,1)).*10.\^{} (D*rand(s,1))},
$$ 
as used in \cite{Wang-Yan-Rahimi-Lou:IEEESP:2020}. The setting $D=3, 5$ and $7$ correspond to $\Theta \approx 10^3, 10^5$ and $10^7$, respectively. As demonstrated in \cite{lorenz2012constructing}, the dynamic range of the signal is a factor that influences the recovery performance. Therefore, we examine sparse signals with various dynamic ranges. For $\mA$ being a Gaussian random matrix, all $s$ nonzero entries of the sparse signal follow the Gaussian distribution $\mathcal{N}(0,1)$.

This section is divided into four subsections. In subsection \ref{subsec:decrease}, we numerically examine the values of the objective function for Problem~\eqref{model-Trad-FP-square-epsilon} across iterations in Algorithm~\ref{alg:ADLPMM} through D-QP and D-LPMM. These experiments help determine the appropriate maximum number of iterations for  D-QP and D-LPMM. Subsection \ref{subsec:noisefree} presents a numerical comparison of Algorithm~\ref{alg:ADLPMM} with existing algorithms for noiseless observations, while subsection \ref{subsec:noise} extends this comparison to scenarios with noisy observations. The final subsection evaluates Algorithm~\ref{alg:ADLPMM} in cases where the sensing matrices are rank-deficient.

\subsection{Objective Function of  Problem~\eqref{model-Trad-FP-square-epsilon} in Algorithm~\ref{alg:ADLPMM}}\label{subsec:decrease}

The primary objective of an algorithm for the $\tau_2$-model is to determine  $\alpha \in [1, n]$ such that $F(\alpha)=0$, where $F$ is defined in \eqref{model-Trad-FP-square-epsilon}. In our proposed Algorithm~\ref{alg:ADLPMM}, the value of $F(\alpha)$ is iteratively estimated by the objective function value $\|\vx^{(k)}\|_1^2-\alpha^{(k-1)}\|\vx^{(k)}\|_2^2$, which approximates $F(\alpha^{(k-1)})$.  A smaller value of this expression indicates better performance for Algorithm~\ref{alg:ADLPMM}. We plot the values of  $\|\vx^{(k)}\|_1^2-\alpha^{(k-1)}\|\vx^{(k)}\|_2^2$ against the iteration number $k$ in Figure~\ref{fig:Obj}. 

In the noiseless case, where $\epsilon=0$ in \eqref{model-Trad-FP-square-epsilon},  Figure~\ref{fig:Obj}(a) uses the 
oversampled DCT sensing matrix with parameters $E=1$, $D=3$ and $s=15$, while  Figure~\ref{fig:Obj}(b) uses the Gaussian sensing matrix with parameters $r=0.2$ and $s=10$. We observed that the value of $F$ by D-QP quickly drops to a small number within just two iterations, whereas  the value of $F$ using D-LPMM gradually decreases toward zero.  These results demonstrate the convergence of the proposed algorithms.  

In the noisy case, where $\epsilon>0$ in \eqref{model-Trad-FP-square-epsilon}, Figure~\ref{fig:Obj} (c) uses the oversampled DCT detection matrix with parameters $E=5$, $D=2$ and $s=8$, and Figure~\ref{fig:Obj} (d) uses the Gaussian detection matrix with parameters $r=0.2$ and $s=8$. The results are similar to those observed in the noiseless case. These experiments further suggest using a small number of iterations for D-QP and a relatively larger number of iterations for D-LPMM.

\begin{figure}[h]
\centering
\begin{tabular}{cc}
\includegraphics[width=2.2in]{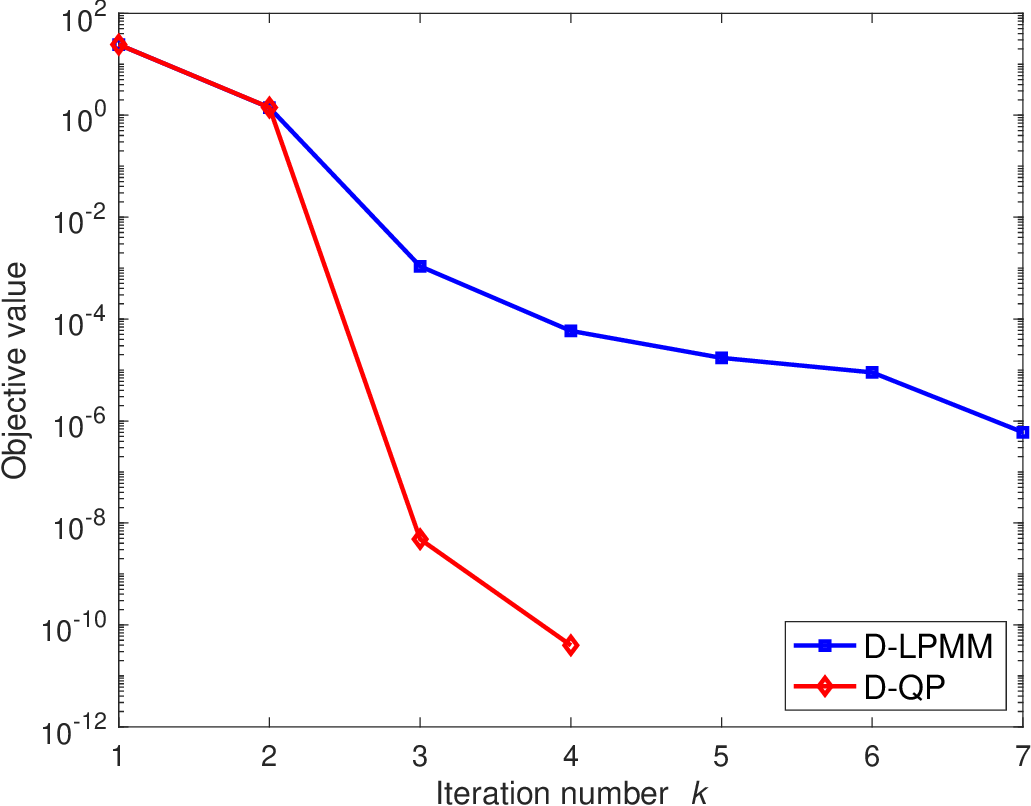}&
\includegraphics[width=2.2in]{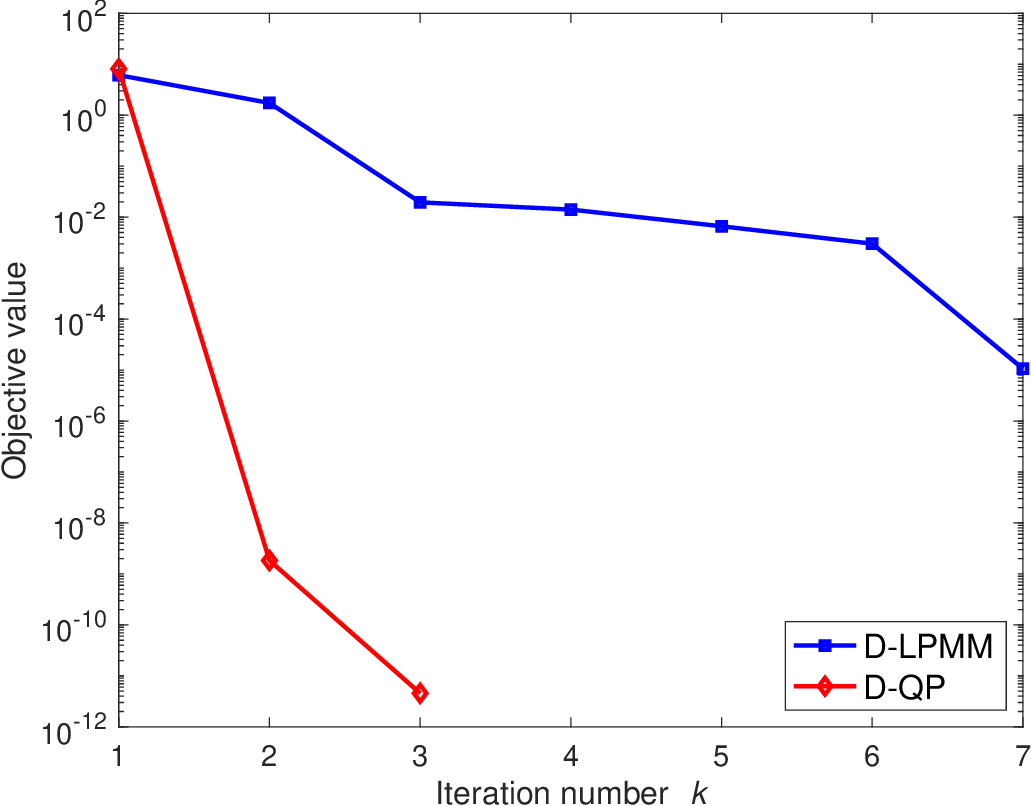}\\
(a)&(b)\\[0.1in]
\includegraphics[width=2.2in]{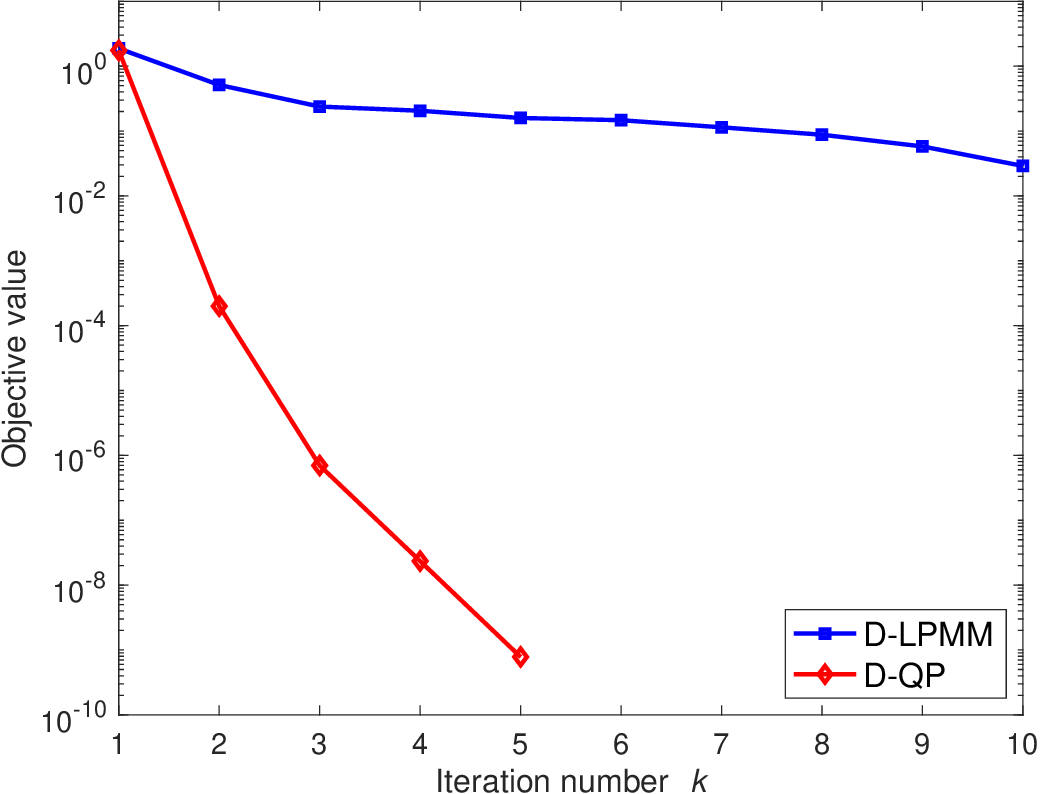}&
\includegraphics[width=2.2in]{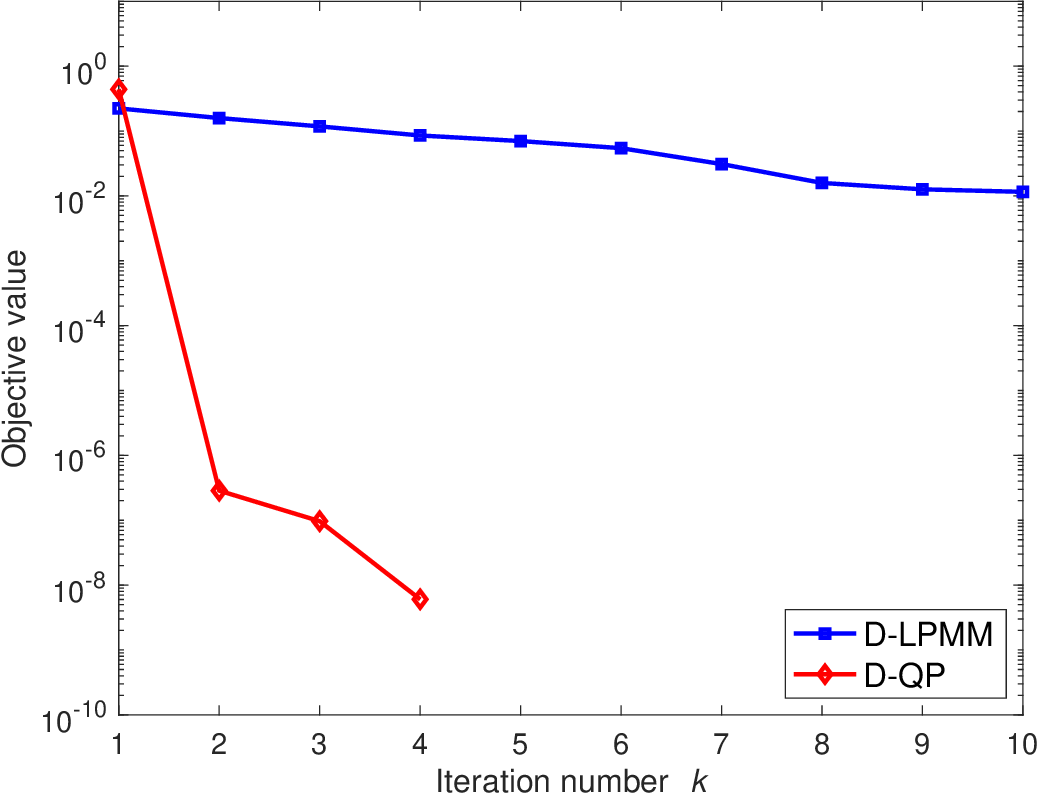}\\
(c)&(d)
\end{tabular}
\caption{Values of $\|\vx^{(k)}\|_1^2-\alpha^{(k-1)}\|\vx^{(k)}\|_2^2$ against the iteration number $k$. (a) Noiseless measurements using an oversampled DCT matrix with $E=1$, $D=3$, $s=15$; (b) Noiseless measurements using a Gaussian matrix with $r=0.2$, $s=10$; (c) Noisy measurements using an oversampled DCT matrix with $E=5$, $D=2$, $s=8$; (d) Noisy measurements using Gaussian matrix with $r=0.2$, $s=8$.}
\label{fig:Obj}
\end{figure}

\subsection{Algorithm Comparison: Noiseless Measurements}\label{subsec:noisefree}

This subsection presents the experimental results from noiseless measurements using the proposed D-QP and D-LPMM algorithms, both based on the $\tau_2$-model. We conduct a comparative analysis against the $L_1/L_2$-A1 and $L_1/L_2$-A2 algorithms, as proposed in \cite{Wang-Yan-Rahimi-Lou:IEEESP:2020}, which are designed for sparse signal recovery using the $\ell_1/\ell_2$-model.

Given the nature of  both the $\tau_2$-model and the $\ell_1/\ell_2$-model, the initial guess $\vx^{(0)}$ for an algorithm significantly influences the final result. For all algorithms tested, we use the $\ell_1$ solution obtained by Gurobi as the initial guess. Each algorithm terminates when the relative error between $\vx^{(k)}$ and $\vx^{(k-1)}$ is less than $10^{-6}$. 

In noiseless experiments, we assess the performance of sparse signal recovery algorithms from two perspectives: success rate and user-friendly implementation. 

The success rate is calculated as the number of successful trials divided by the total number of trials. A trial is considered successful if the relative error between the ground truth vector $\vx$ and the reconstructed solution $\vx^*$, i.e., ${\|\vx^*-\vx\|_2}/{\|\vx\|_2}$ is less than $10^{-3}$.  Figure \ref{fig:Alg-comp-DCT} illustrates the success rate for the oversampled DCT sensing matrix case with $E\in\{1, 10, 20\}$ and $D\in \{3, 5, 7\}$. Our proposed algorithms show performance comparable to $L_1/L_2$-A1 and $L_1/L_2$-A2. Consistent with the findings in \cite{Rahimi-Wang-Dong-Lou:SIAMSC:2019}, we observe that increased coherence results in enhanced sparse recovery performance. Furthermore, as discussed in \cite{Wang-Yan-Rahimi-Lou:IEEESP:2020}, algorithms that are scale-invariant generally exhibit improved success rates when operating across higher dynamic ranges. Similarly, Figure \ref{fig:Alg-comp-Gaussian} displays the success rate for the case of Gaussian sensing matrix with $r\in \{0.2, 0.5, 0.8\}$. In this case as well, all four algorithms demonstrate comparable success rates, each surpassing the recovery performance of the $\ell_1$ method in terms of success rate. Notably, the D-QP algorithm consistently achieves the best or second-best performance compared to other algorithms.

\begin{figure}[htp]
    \centering
    \begin{subfigure}{0.3\textwidth}
        \includegraphics[width=\linewidth]{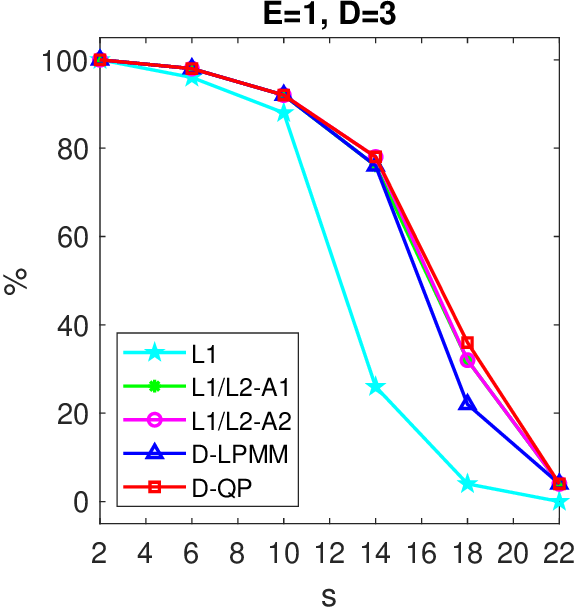}
    \end{subfigure}
    \hfill
    \begin{subfigure}{0.3\textwidth}
        \includegraphics[width=\linewidth]{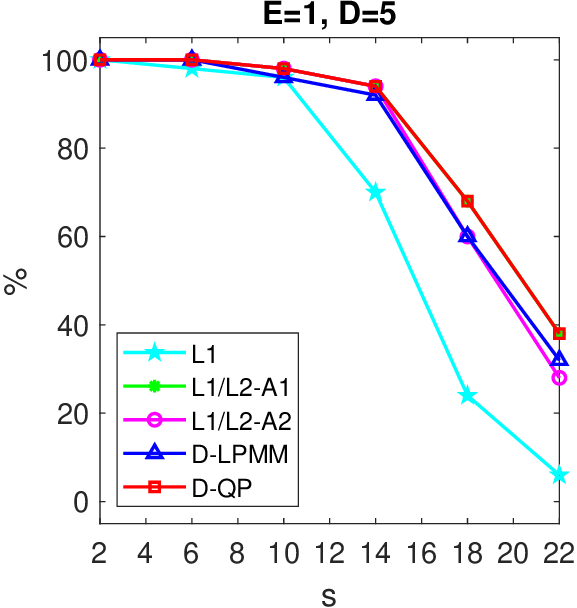}
    \end{subfigure}
    \hfill
    \begin{subfigure}{0.3\textwidth}
        \includegraphics[width=\linewidth]{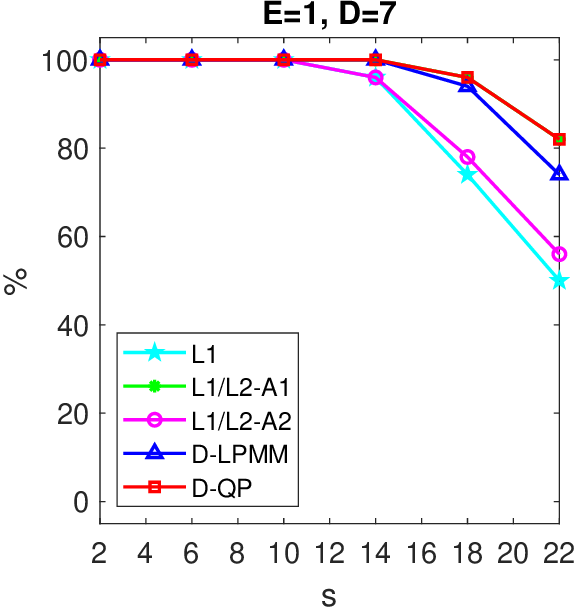}
    \end{subfigure}

    \medskip

    \begin{subfigure}{0.3\textwidth}
        \includegraphics[width=\linewidth]{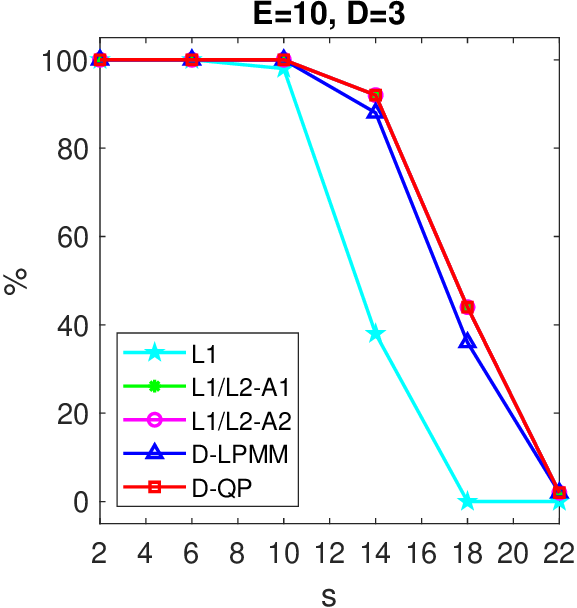}
    \end{subfigure}
    \hfill
    \begin{subfigure}{0.3\textwidth}
        \includegraphics[width=\linewidth]{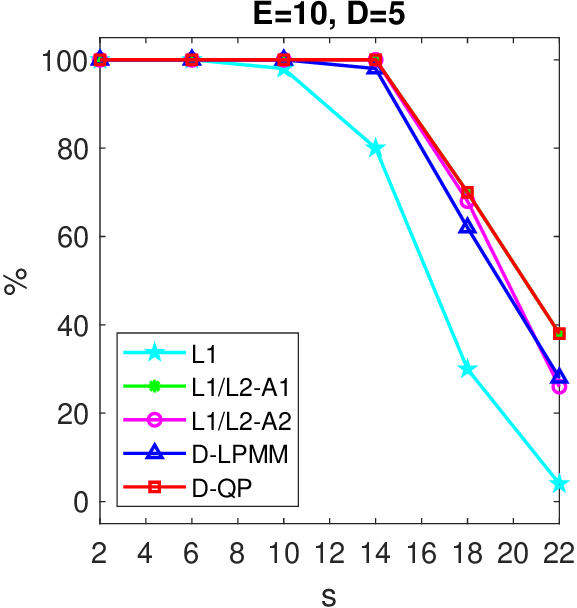}
    \end{subfigure}
    \hfill
    \begin{subfigure}{0.3\textwidth}
        \includegraphics[width=\linewidth]{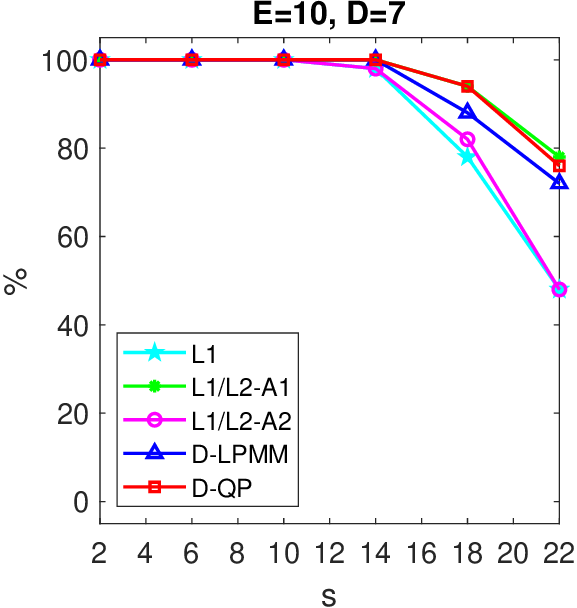}
    \end{subfigure}

    \medskip

    \begin{subfigure}{0.3\textwidth}
        \includegraphics[width=\linewidth]{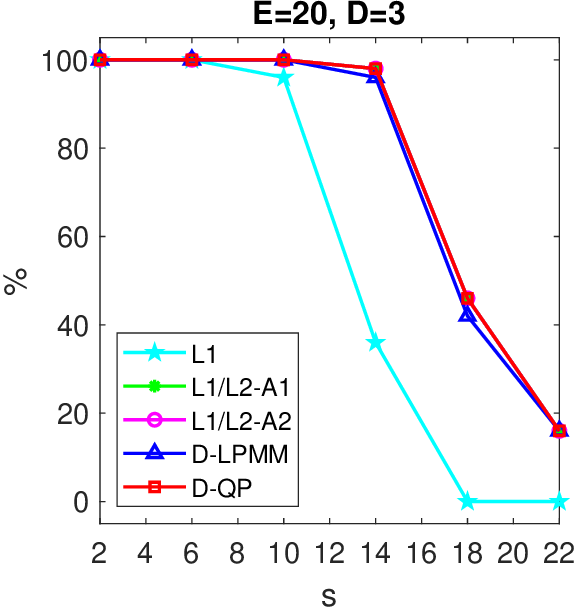}
    \end{subfigure}
    \hfill
    \begin{subfigure}{0.3\textwidth}
        \includegraphics[width=\linewidth]{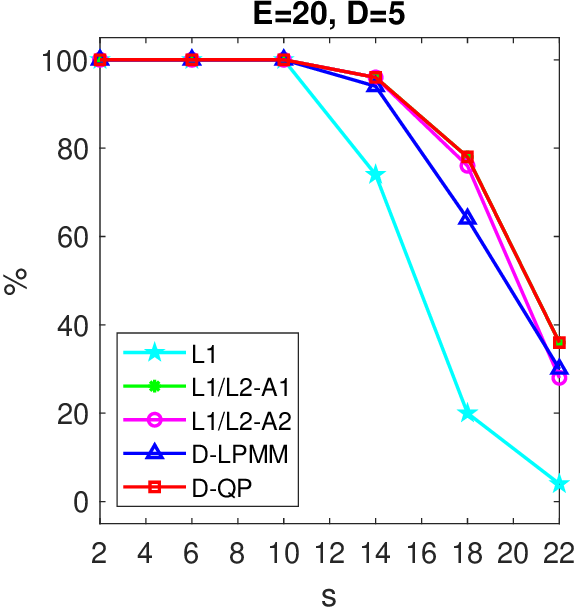}
    \end{subfigure}
    \hfill
    \begin{subfigure}{0.3\textwidth}
        \includegraphics[width=\linewidth]{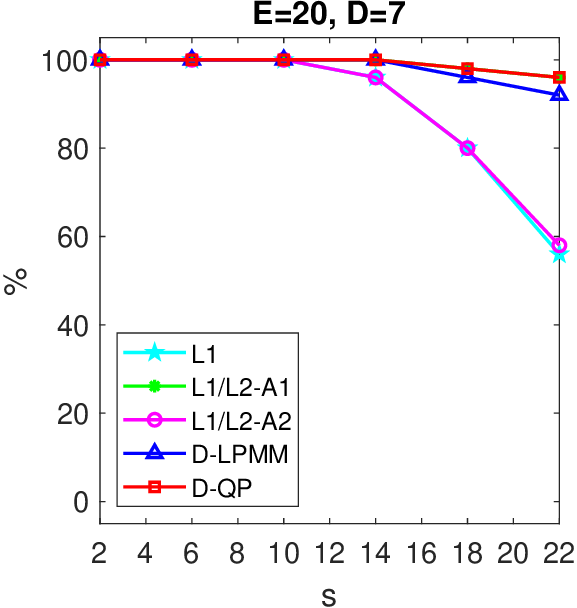}
    \end{subfigure}
    \caption{Algorithmic comparison of success rates is conducted for  DCT sensing matrices. The presentation of results is structured in a $3 \times 3$ grid format, where rows correspond to $E=1,10,20$ from top to bottom, and columns correspond to $D=3,5,7$ from left to right.}

    \label{fig:Alg-comp-DCT}
\end{figure}

For the user-friendly implementation, we emphasize two crucial factors. The first is the number of hyperparameters required by each algorithm. Algorithms that require fewer hyperparameters simplify the tuning process, enhancing accessibility, particularly for individuals with limited expertise in parameter optimization. This simplicity is vital as it broadens the algorithm's applicability across diverse scenarios without requiring intricate customization. The second factor is the computational complexity of the algorithms, primarily computational time. Computationally efficient algorithms are preferable in practical scenarios, capable of processing large datasets effectively and suitable for applications with constrained computational resources. These factors are essential in evaluating the algorithms' practicality and user-friendliness, especially in real-world applications where a balance between accuracy and efficiency is crucial.
\begin{figure}[h]
    \centering
    \begin{subfigure}{0.3\textwidth}
        \includegraphics[width=\linewidth]{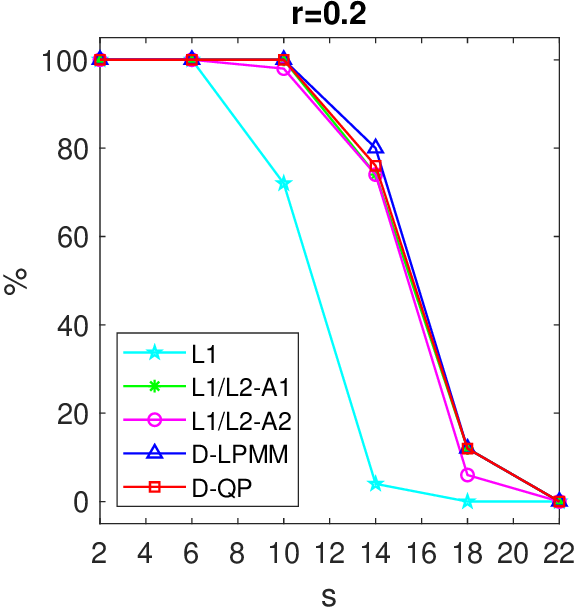}
    \end{subfigure}
    \hfill
    \begin{subfigure}{0.3\textwidth}
        \includegraphics[width=\linewidth]{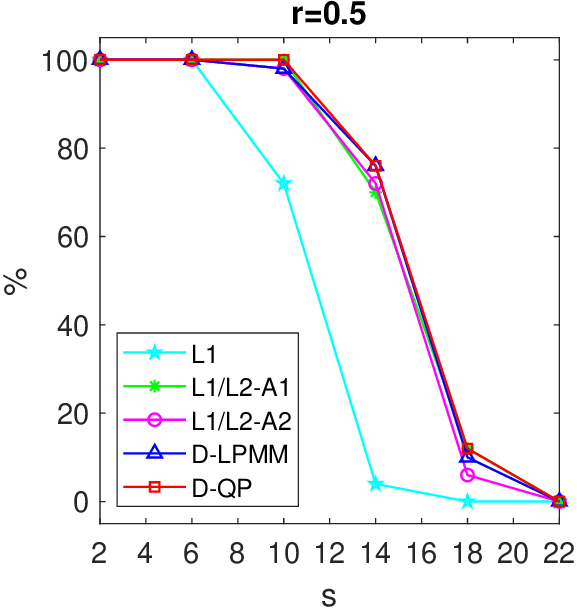}
    \end{subfigure}
    \hfill
    \begin{subfigure}{0.3\textwidth}
        \includegraphics[width=\linewidth]{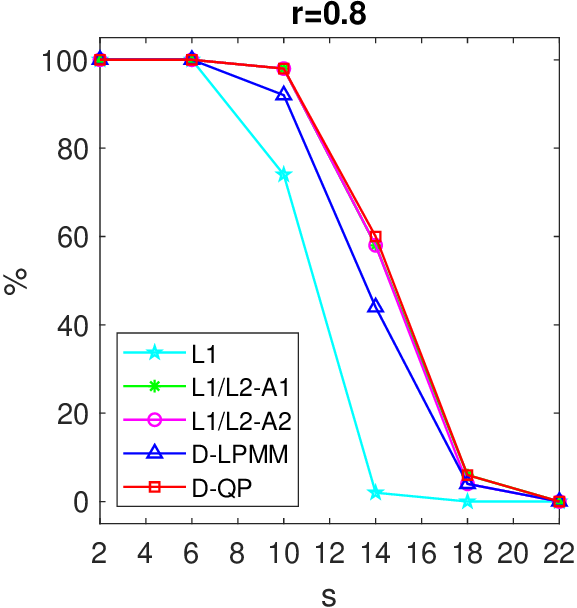}
    \end{subfigure}

    \caption{Algorithmic comparison of success rates is conducted for the Gaussian sensing matrices, with figures arranged from left to right for $r = 0.2, 0.5, 0.8$.}
    \label{fig:Alg-comp-Gaussian}
\end{figure}

The first pivotal factor we discussed for the user-friendly implementation is hyperparameters: D-LPMM was configured with a single parameter, $\rho$, set to 100 for the oversampled DCT sensing matrices and 2 for Gaussian matrices. In contrast, D-QP, grounded in quadratic programming, operates efficiently without the need for parameter tuning. In implementing the $L_1/L_2$-A2 approach, we adhered to the recommended configurations for the parameters $\beta$ and $\rho$ in \cite{Wang-Yan-Rahimi-Lou:IEEESP:2020}. However, $L_1/L_2$-A1, despite its linear programming base and apparent lack of parameters, is not without challenges. As previously discussed in Section~\ref{sec:prel}, it is susceptible to scenarios where the objective function becomes unbounded, resulting in the absence of a viable solution. A further limitation of $L_1/L_2$-A1 is its lack of a defined convergence analysis, casting doubt on its reliability in consistently reaching a critical point. The details of the parameter configurations used in our study are cataloged in Table~\ref{tab:parameter_setup}. This underscores the practical superiority of our proposed methods: D-QP's independence from parameterization and D-LPMM's minimalistic parameter requirements, which markedly simplify their usage compared to other algorithms that necessitate intricate parameter adjustments for different types of sensing matrices.

\begin{table}[ht]
\caption{Parameters setup for all testing algorithms.}
\label{tab:parameter_setup}
\centering
\begin{tabular}{c|cc}
\noalign{\hrule height 1pt}
Algorithm       & Oversampled DCT & Gaussian \\ \hline
\noalign{\hrule height 0.8pt}
D-QP          & Parameter-free  & Parameter-free  \\ 
D-LPMM        & \(\rho = 100\) & \(\rho = 2\) \\ 
$L_1/L_2$-A1 \cite{Wang-Yan-Rahimi-Lou:IEEESP:2020}         & Parameter-free & Parameter-free \\ 
$L_1/L_2$-A2 \cite{Wang-Yan-Rahimi-Lou:IEEESP:2020}         & $\beta=1$, $\rho=20$  & $\beta=10^{-5}$, $\rho=0.3$  \\ \hline
\noalign{\hrule height 1pt}
\end{tabular}
\end{table}

The second critical aspect of the user-friendly implementation is computational time. Table \ref{tab:average_computation_time} shows the computational time for the oversampled DCT sensing matrix with $(E,D)=(1, 3), (10, 5)$, and the Gaussian sensing matrix with $r=0.2, 0.5$. In terms of computational efficiency, both D-LPMM and D-QP show comparable performance, effectively balancing speed and ease of use. D-LPMM, with its minimal parameter tuning, and D-QP, operating without any parameters, both demonstrate swift processing capabilities. 

\begin{table}[ht] 
\caption{Average computational time (in seconds) for sparse signal recovery. }
\label{tab:average_computation_time}
\centering
\resizebox{\textwidth}{!}
{
\begin{tabular}{ccccc|cccc}
\hline
\noalign{\hrule height 1pt}
$s$       & $L_1/L_2$-A1 & $L_1/L_2$-A2 & D-LPMM & D-QP & $L_1/L_2$-A1 & $L_1/L_2$-A2 & D-LPMM & D-QP\\ 
\noalign{\hrule height 1pt}
&\multicolumn{4}{c|}{Oversampled DCT with $(E, D)= (1, 3)$} &\multicolumn{4}{c}{Oversampled DCT with $(E,D)=(10,5)$}\\ 
\hline
2       & 0.1777 & 0.1951 & 0.1419 & 0.7693 & 0.1857 & 0.1687 & 0.1497 & 0.8255\\
6       & 0.2004 & 0.1742 &	0.1939 & 0.8696 & 0.2012 & 0.1380 & 0.1903 & 0.8130\\
10      & 0.2245 & 0.2106 &	0.2417 & 0.9159 & 0.2214 & 0.1511 & 0.2340 & 0.8302\\
14      & 0.3297 & 0.3239 &	0.4639 & 1.5123 & 0.3167 & 0.1827 &	0.3385 & 1.6122\\
18      & 0.4993 & 1.3699 &	2.1894 & 3.3791 & 0.4228 & 0.2957 & 0.7576 & 3.9726\\
22      & 0.6525 & 2.7971 &	3.7691 & 6.1085 & 0.5427 & 0.6337 &	1.3844 & 6.5536\\
\noalign{\hrule height 0.5pt}
&\multicolumn{4}{c|}{Gaussian matrix with $r= 0.2$} &\multicolumn{4}{c}{Gaussian matrix with $r=0.5$}\\ 
\hline
2       & 0.1929 & 0.1309 & 0.2317 & 0.8362 & 0.1901 & 0.1284 & 0.2537 & 0.8299\\
6       & 0.1977 & 0.1342 &	0.4330 & 0.8376 & 0.1948 & 0.1361 & 1.0063 & 0.8348\\
10      & 0.2359 & 0.1968 &	0.6730 & 1.0305 & 0.2337 & 0.1984 & 1.5466 & 1.0345\\
14      & 0.6772 & 0.7952 &	2.0967 & 3.6477 & 0.5620 & 0.7556 &	3.1580 & 2.7525\\
18      & 1.4552 & 1.1727 &	4.3635 & 8.8977 & 1.3572 & 1.1323 & 4.7501 & 9.0134\\
22      & 1.9751 & 1.1301 &	4.7333 & 13.6945 & 1.8108 & 1.1501 & 4.8884 & 12.2452\\
\noalign{\hrule height 1pt}
\end{tabular}
}
\end{table}

\subsection{Algorithm Comparison: Noisy Measurements} \label{subsec:noise}
This subsection addresses the challenge of sparse signal recovery in the presence of white Gaussian noise using the proposed D-QP and D-LPMM algorithms, both based on the $\tau_2$-model.  We perform a comparative analysis against the MBA method described in \cite{Zeng-Yi-Pong:SIAMOP:2021}, which focuses on recovering sparse signals from noisy measurements via the $\ell_1/\ell_2$-model.

The experimental setup and test instances mirror those detailed in  \cite{Zeng-Yi-Pong:SIAMOP:2021}. Specifically, the oversampled DCT sensing matrix $\mA$ and the ground truth signal $\vx$ replicate noiseless experiment configurations. Observations $\vb$ are generated as $\vb = \mA\vx + 0.01 \xi$, where $\xi \in \R^m$ represents a vector with i.i.d. standard Gaussian entries. We define $\epsilon = 1.2 \|0.01\xi\|$.

The initial points for the algorithms compared in this subsection are identical to those used in \cite{Zeng-Yi-Pong:SIAMOP:2021} and are given by:
\begin{eqnarray*}
    \vx^{(0)} = \begin{cases}
        \mA^\dag\vb + \epsilon \frac{\vx_{\ell_1}-\mA^\dag \vb}{\|\mA \vx_{\ell_1}-\vb\|_2}, \quad\quad & \text{if }\|\mA\vx_{\ell_1}-\vb\|_2 > \epsilon \\
        \vx_{\ell_1} \quad & \text{otherwise,}
    \end{cases}
\end{eqnarray*}
where $\vx_{\ell_1}$ is computed using SPGL1 \cite{BergFriedlander:SIAMSC:2008} (version 2.1) with default settings, and $\mA^\dag$ is the pseudo-inverse of $\mA$. Note that $\vx^{(0)}$ meets the constraints of the problem. Experimentally, we vary $s\in\{4, 8, 12\}$, $E\in\{5, 15\}$, and $D\in\{2, 3\}$.

For MBA and D-LPMM, termination occurs when the relative error between successive iterations, $\vx^{(k)}$ and $\vx^{(k-1)}$, falls below $10^{-6}$, or when the iteration count exceeds $5n$. For D-QP, the criteria are aligned with the other methods for relative error, but differ in that termination occurs after fewer iterations, specifically 5, as suggested in subsection~\ref{subsec:decrease}. The parameters of the MBA method follow the recommendations specified in \cite{Zeng-Yi-Pong:SIAMOP:2021}.  For D-LPMM, we set $\beta=\rho=80$. In particular, D-QP operates without parameter dependencies.

Similarly to the noiseless experiments, we assess the algorithms' performance from two perspectives: relative error and user-friendly implementation. Table~\ref{tab:relerr-comp-noisy} presents the relative errors of the reconstructed solution $\vx^*$ to the ground truth $\vx$ averaged over 20 trials, i.e., $\|\vx^*-\vx\|_2/\|\vx\|_2$, along with computation times. 

Both the algorithms based on the $\ell_1/\ell_2$-model and the $\tau_2$-model exhibit lower recovery errors compared to the $\ell_1$ method. The algorithms D-LPMM, D-QP, and MBA yield comparable results in terms of relative error, with D-LPMM excelling in scenarios involving high coherence sensing matrices and high dynamic range signals.

Regarding user-friendly implementation, all three algorithms show comparable computation times. D-LPMM proves to be the fastest particularly for high coherence matrices and high dynamic range signals. In terms of parameters, D-QP operates efficiently without the need for parameter tuning. D-LPMM showcases robust performance even with a simplified parameter configuration. Although the algorithm technically has two parameters, $\beta$ and $\rho$, empirical evidence suggests that setting both parameters to the same value yields excellent results. Therefore, it effectively operates with a single adjustable parameter. Our methods, showcasing D-QP's independence from adjustments and D-LPMM's effective single-parameter setup, feature streamlined parameter settings that ensure easy adaptation across various sensing matrices.

\begin{table}[ht] 
\caption{Algorithm comparison for Gaussian noise signal recovery}
\centering
\resizebox{\textwidth}{!}
{
\begin{tabular}{ccc|cccc|cccc}
\hline
\noalign{\hrule height 1pt}
\multicolumn{3}{c|}{settings} &\multicolumn{4}{c|}{Relative error} &\multicolumn{4}{c}{Computational time (s)}\\ 
\hline
s       & E & D & spg$\ell_1$ & MBA & D-LPMM & D-QP & spg$\ell_1$ & MBA & D-LPMM & D-QP \\ 
\noalign{\hrule height 1pt}
\hline
4 & 5 & 2 & 5.165e-03  & 4.043e-03 & 4.074e-03 & 4.309e-03 & 0.04 & 0.01 & 0.11 & 5.12 \\
4 & 5 & 3 & 1.098e-03  & 9.323e-04 & 9.302e-04 & 2.708e-03 & 0.02 & 0.01 & 0.09 & 6.48  \\
4 & 15 & 2 & 3.978e-01  & 7.841e-02 & 7.841e-02 & 5.910e-02 & 0.06 & 0.57 & 0.17 & 6.48  \\
4 & 15 & 3 & 3.003e-01  & 9.362e-02 & 2.208e-03 & 3.709e-03 & 0.08 & 2.33 & 0.22 & 8.07 \\
\noalign{\hrule height 0.5pt}
\hline
8 & 5 & 2 & 3.247e-02  & 2.288e-03 & 2.313e-03 & 2.356e-03 & 0.06 & 0.11 & 0.57 & 5.93  \\
8 & 5 & 3 & 4.536e-03  & 6.573e-04 & 6.226e-04 & 4.088e-03 & 0.05 & 0.08 & 0.79 & 6.83  \\
8 & 15 & 2 & 4.599e-01  & 1.383e-01 & 1.499e-01 & 1.549e-01 & 0.07 & 1.99 & 1.09 & 7.41  \\
8 & 15 & 3 & 3.809e-01  & 2.891e-01 & 5.298e-02 & 9.629e-02 & 0.06 & 2.84 & 1.47 & 8.69  \\
\noalign{\hrule height 0.5pt}
12 & 5 & 2 & 1.340e-01  & 5.199e-02 & 3.639e-02 & 6.085e-02 & 0.05 & 0.82 & 1.53 & 6.16  \\
12 & 5 & 3 & 5.814e-02  & 3.855e-02 & 3.739e-03 & 7.724e-02 & 0.06 & 0.85 & 1.69 & 6.47 \\
12 & 15 & 2 & 5.208e-01  & 2.028e-01 & 1.880e-01 & 1.975e-01 & 0.07 & 3.32 & 2.57 & 8.25  \\
12 & 15 & 3 & 5.262e-01  & 5.059e-01 & 3.725e-01 & 1.195e-00 & 0.07 & 2.76 & 1.85 & 8.34  \\
\noalign{\hrule height 1pt}
\end{tabular}
}
\label{tab:relerr-comp-noisy}
\end{table}

\subsection{Rank-deficient Sensing Matrices}\label{subsec:Rank}
In this subsection, we evaluate the performance of the proposed algorithms, D-QP and D-LPMM, with rank-deficient sensing matrices. Although the theoretical analysis assumes that the sensing matrix $\mA$ is full-rank, this assumption may appear to significantly constrain the applicability of the method. However, it is important to note that, in practice, the proposed Algorithm~\ref{alg:ADLPMM} remains effective even for rank-deficient sensing matrices. This is because the computations involved in the quadratic programming and D-LPMM algorithms do not require $\mA$ to be full rank. Note that the $L1/L2$-A2 algorithm requires the full rankness of $\mA$. 

In our experiment, we generate a $69 \times 1024$ sensing matrix $\mA$ in two different ways. First, we create an oversampled DCT matrix $64 \times 1024$ with $E=10$. Then, we augment this matrix in rows by randomly selecting $5$ rows from the same matrix or using linear combinations of these selected rows.

Figure~\ref{fig:rank-deficient} displays the success rate in $50$ tests for the noiseless case, demonstrating that both the D-QP and D-LPMM algorithms perform well, even when the sensing matrices are rank-deficient in these test cases.

\begin{figure}[h]
\centering
\begin{tabular}{cc}
\includegraphics[width=2.2in]{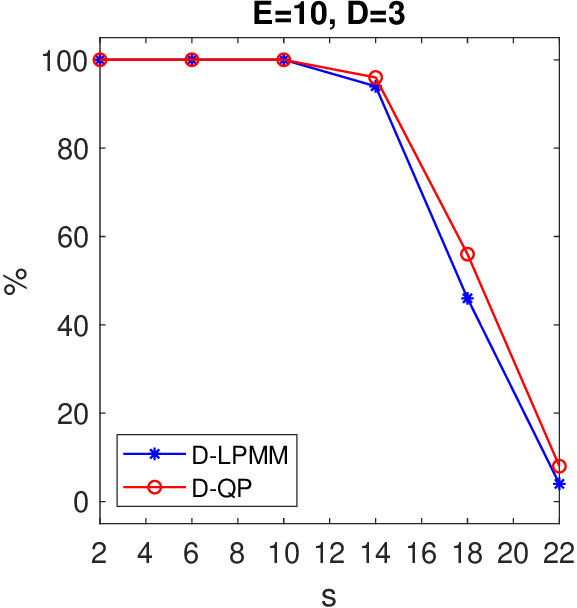}&
\includegraphics[width=2.2in]{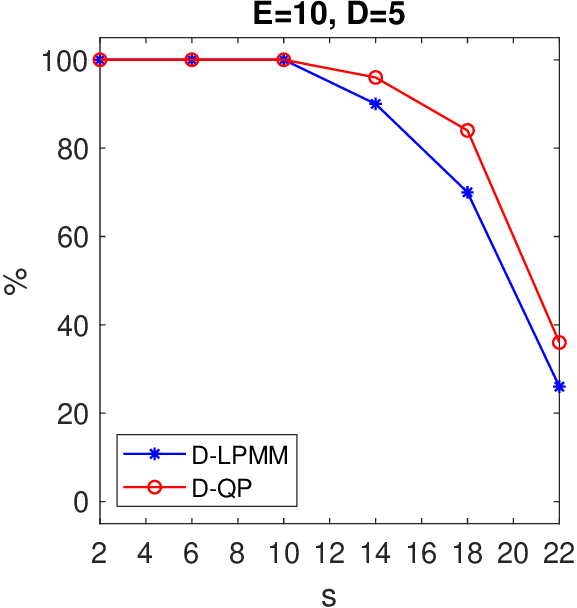}
\end{tabular}
\caption{The success rates of D-QP and D-LPMM using rank-deficient oversampled DCT sensing matrices of size $69 \times 1024$ with $E=10$. The tested sparse signals are generated with $D=3$ and $D=5$. (Left) Results where the last $5$ rows of $\mA$ are randomly selected from the first $64$ rows. (Right) Results where the last $5$ rows of $\mA$ are linear combinations of randomly selected rows from the first $64$ rows.
}
\label{fig:rank-deficient}
\end{figure}

\section{Conclusion}\label{sec:conclusion}
In this work, we introduce the $\tau_2$-model, a novel approach to sparse signal recovery, which effectively addresses the limitations of the $\ell_0$ norm models utilizing the square of $\ell_1/\ell_2$ norms. Our study was centered on a thorough exploration of the properties of the $\tau_2$-model, delving deep into the intricacies of the corresponding optimization problem, fundamentally rooted in Dinkelbach's procedure. Through rigorous theoretical analysis and numerical experiments, we have validated the model's effectiveness in sparse signal recovery. Our numerical experiments highlighted the impact of algorithms developed for the $\tau_2$-model on the recovered signals.

In our future work, we plan to explore several promising projects. First, our objective is to conduct a more comprehensive study by improving D-QP and D-LPMM and incorporating other fractional programming algorithms, such as those suggested in \cite{Li-Shen-Zhang-Zhou:ACHA-2022,Zhang-Li:SIAMOP:2022}. Second,  we plan to investigate a general optimization model as follows:
\begin{equation}\label{model:square:q}
\arg\inf\left\{\frac{\|\vx\|_1^q}{\|\vx\|_2^q}: \|\mA \vx- \vb\|_2 \le \epsilon, \; \vx\in \mathbb{R}^n\right\}
\end{equation}
for any $q\ge 1$. Clearly, the set of solutions for model~\eqref{model:square:q} is identical to that for the ${\ell_1/\ell_2}$-model and the $\tau_2$-model. Theoretically, for $q>1$  model~\eqref{model:square:q} can be solved using the same procedure as for $q=2$. Since the proximity operator of $\|\cdot\|_1^q$ can be efficiently computed for $q=2,3,4$ (see our previous work in [22]), the corresponding optimization can be efficiently solved. We expect results similar to those in Theorem~\ref{thm:convergence} for the algorithm applied to this general model.

\begin{bluetext}

\end{bluetext}

\section*{Acknowledgement} 
The work of L. Shen was supported in part by the National Science Foundation under grant DMS-2208385 and by the Air Force Summer Faculty Fellowship Program (SFFP). Approved for public release on February 26, 2024, case number: AFRL-2024-1037. The authors gratefully acknowledge the constructive comments and suggestions from two reviewers, which significantly enhanced the quality of this paper.

\section*{Conflict of interest}
The authors declare that they have no conflict of interest.  Any opinions, findings and conclusions or recommendations expressed in this material are
those of the authors and do not necessarily reflect the views of AFRL (Air Force Research Laboratory).

\bibliographystyle{spmpsci}
\bibliography{shen_jia_2023}

\end{document}